\documentclass[12pt,a4paper]{amsart}
\usepackage{amsmath,mathrsfs,amssymb,amsthm, latexsym, amsfonts,alltt}
\usepackage[T1]{fontenc}
\usepackage{xy}
\usepackage{xr}
\usepackage{enumerate}
\usepackage[latin1]{inputenc}
\usepackage[francais, english]{babel}
\include{xy}
\xyoption{all}
\numberwithin{equation}{section}

\nopagebreak
\newtheorem{theo}{Th\'eor\`eme}[]
\theoremstyle{plain}
\newtheorem{theorem}{Th\'eor\`eme}[section]
\newtheorem{corollary}[theorem]{Corollaire}
\newtheorem{lemma}[theorem]{Lemme}
\newtheorem{proposition}[theorem]{Proposition}
\theoremstyle{definition}
\newtheorem{defi}[theorem]{D\'efinition}
\newtheorem{remark}[theorem]{Remarque}

\DeclareMathOperator{\supp}{supp}

\DeclareMathOperator{\Irr}{Irr}

\DeclareMathOperator{\Hom}{Hom}

\DeclareMathOperator{\rang}{rang}

\DeclareMathOperator{\Jac}{Jac}
\DeclareMathOperator{\ind}{ind}

\DeclareMathOperator{\ry}{Nrd}
\DeclareMathOperator{\trd}{trd}

\author[Alberto M\'inguez]{Alberto M\'inguez}\thanks{Partially supported by MTM2004-07203-C02-01 and FEDER}
\address{Alberto M\'inguez, Laboratoire de Math\'ematiques, Universit\'e Paris-Sud, B\^at 425
91405 Orsay Cedex, France, CNRS UMR 8628. \\ URL: {\rm  http://www.math.u-psud.fr/$\sim$minguez/}} 
\email{minguez@clipper.ens.fr}

\begin{document}

\title[Correspondance de Howe: paires duales de type II]{Correspondance de Howe explicite: paires duales de type II\\ \vspace{1cm} \textnormal{Explicit Howe correspondence: dual pairs of  type II}}

\begin{otherlanguage}{english}
\begin{abstract}
In this article, we give a new method for proving Howe correspondence in the case of dual pairs of type $\left( {\rm GL}_n, {\rm GL}_m \right)$ over a non-Archimedean locally compact field $F$. The proof consists in combining a study on Kudla's filtration \cite{kudla} with the results of \cite{Min1} about the irreducibility of a parabolically induced representation. The proof is valid for $F$ of any characteristic and allows us to make the correspondence explicit in terms of Langlands parameters. Hence it generalizes the results of  \cite{wat} and answers completely all questions studied in \cite{Mu1} and \cite{Mu2} for dual pairs of type II.
\end{abstract}
\end{otherlanguage}
\maketitle
\selectlanguage{francais}
\begin{abstract}
Dans cet article, nous proposons une nouvelle m\'ethode pour d\'emontrer la bijectivit\'e de la correspondance de Howe pour les paires duales du type $\left( {\rm GL}_n, {\rm GL}_m \right)$ sur un corps $F$ localement compact non archim\'edien. La preuve est bas\'ee sur une \'etude soigneuse  de la filtration de Kudla \cite{kudla} ainsi que sur les r\'esultats de \cite{Min1} \`a propos de l'irr\'eductibilit\'e d'une repr\'esentation induite parabolique. Elle est valable pour $F$ de caract\'eristique quelconque et nous permet d'expliciter la bijection en termes des param\`etres de Langlands. Elle  g\'en\'eralise donc les r\'esultats de \cite{wat} et r\'epond totalement aux questions \'etudi\'ees dans \cite{Mu1} et \cite{Mu2} pour les paires duales de type II. 
\end{abstract}
{\bf Codes MSN:} 11F27, 22E50.
\vspace{1cm}
\section*{Introduction}
Soit $F$ un corps commutatif localement compact non archim\'{e}dien de caract\'{e}ristique
r\'{e}siduelle $p>0$. Soit $\psi: F \rightarrow \mathbb{C}^\times$ un caract\`ere additif non trivial de $F$. Si $W$ est un espace vectoriel symplectique sur $F$, de dimension 
finie, on dispose du groupe m\'etaplectique $\widetilde{Sp}\left(W\right)$, qui est  un rev\^etement \`a deux feuillets du groupe symplectique $Sp\left(W\right)$, et d'une repr\'esentation $\left(\omega, S\right)$ de $\widetilde{Sp}\left(W\right)$ canoniquement attach\'ee \`a $\psi$, dite repr\'esentation de Weil ou m\'etaplectique, sur un espace de fonctions $S$ \`a valeurs complexes. Soit $\left(G_1,G_2\right)$ une paire duale r\'eductive (\textit{cf.} \cite[1.I.17]{MVW}) dans $Sp\left(W\right)$: ou bien $\left(G_1,G_2\right)$ est une paire de groupes classiques -symplectique, orthogonal, unitaire- (paires duales de type I) ou bien une paire de groupes lin\'eaires (paires duales de type II). Notons $\widetilde{G_1}$ et $\widetilde{G_2}$ leurs images r\'eciproques dans $\widetilde{Sp}\left(W\right)$.

Soit $\pi$ une repr\'esentation lisse irr\'eductible de $\widetilde{G_1}$ quotient de $\omega$. Notons $S\left[\pi\right]$ le plus grand quotient $\pi$-isotypique de $\omega$. Il est de la forme
$$S\left[\pi\right]=\pi_1 \otimes \Theta\left(\pi\right),$$
en tant que $\widetilde{G_1}\times\widetilde{G_2}$-module, o\`u $\Theta\left(\pi\right)$ est une repr\'esentation lisse de longueur finie de $\widetilde{G_1}$. 

Roger Howe et Jean-Loup Waldspurger \cite{walds}, \cite{MVW} ont prouv\'e que, dans le cas o\`u $p$ est impair et que $(G_1,G_2)$ est de type I, si $\Theta(\pi) \neq 0$, alors $\Theta (\pi)$ a un unique quotient irr\'eductible, not\'e $\theta(\pi)$. L'application $\pi \mapsto \theta(\pi)$ est une bijection entre l'ensemble des repr\'esentation lisses irr\'eductibles $\pi$ de $\widetilde{G_1}$ telles que $\Theta(\pi) \neq 0$ et l'ensemble des repr\'esentation lisses irr\'eductibles $\pi'$ de $\widetilde{G_2}$ telles que $\Theta(\pi') \neq 0$. Elle est appel\'ee la correspondance de Howe. Nous nous proposons de montrer un th\'eor\`eme similaire pour les paires duales de type II, valable pour tout $p$, et d'expliciter, en termes des param\`etres de Langlands, la correspondance $\pi \mapsto \theta(\pi)$, ce qui d\'etermine l'ensemble des repr\'esentations $\pi$ telles que $\Theta(\pi) \neq 0$.

Dans le cas des paires duales de type II, Roger Howe, dans un manuscrit non publi\'e, avait prouv\'e la bijectivit\'e de la correspondance. Notre m\'ethode, diff\'erente, rend, de plus, la correspondance explicite.

Passons \`a une pr\'esentation plus d\'etaill\'ee des r\'esultats:

Soit $D$ une alg\`{e}bre \`{a} division de centre $F$ de dimension finie $d^2$ sur $F$ et soient $n$ et $m$ des entiers strictement positifs. On note $\mathcal{M}_{n,m}$ (resp. $\mathcal{M}_{n}$) l'ensemble des matrices $n \times m$ (resp. $n \times n$) \`a coefficients dans $D$. Le groupe ${\rm GL}_n(D)$ des matrices inversibles dans $\mathcal{M}_{n}$ sera not\'e $G_n$.

On note $S_{n,m}=S\left( \mathcal{M}_{n,m}\right)$ l'espace vectoriel des fonctions $ \Phi$ de $\mathcal{M}_{n,m}$ dans $\mathbb{C}$, localement constantes \`{a} support compact.
La repr\'{e}sentation m\'{e}taplectique $\omega_{n,m}$
restreinte \`{a} la paire duale $G_{n}\times G_{m}$ (\textit{cf.} \cite[3.III.1]{MVW}) est la repr\'esentation
\begin{equation}\label{omega}
\omega_{n,m}(g,g')=\nu(g)^{\frac{-m}{2}}\sigma_{n,m}(g,g')\nu(g')^{\frac{n}{2}},\end{equation}
o\`u on note $\nu=|\ry|_F$, la valeur absolue de la norme r\'eduite et $$\sigma _{n,m}:G_{n}\times
G_{m}\rightarrow {\rm GL}\left( S_{n,m} \right) $$ la
repr\'{e}sentation naturelle de $G_{n}\times G_{m}$ d\'{e}finie par 
\begin{equation*}
\sigma _{n,m}\left( g,g^{\prime }\right) \Phi \left( x\right) =\Phi \left(
g^{-1}xg^{\prime }\right),
\end{equation*}
pour $g\in G_{n}, \ g^{\prime }\in G_{m},\ x\in \mathcal{M}_{n,m},\ \Phi \in
S_{n,m} $.

Le r\'esultat principal de cet article est le th\'eor\`eme suivant.
\begin{theo}\label{theo}{\bf(voir corollaire \ref{expl2}).}
Soit $\pi$ une repr\'esentation irr\'eductible de $G_{n}$.
\begin{enumerate}
\item Si $\Hom_{G_n}\left( \omega_{n,m}, \pi \right) \neq 0$, alors il existe une unique repr\'esentation irr\'eductible $\pi'$ de $G_{m}$ telle que
$$\Hom_{G_n \times G_m}\left( \omega_{n,m}, \pi \otimes \pi'\right) \neq 0.$$ De plus, $\dim \left( \Hom_{G_n \times G_m}\left( \omega_{n,m}, \pi \otimes \pi'\right) \right)=1$.
\item Supposons $n \leq m$. Alors  $\Hom_{G_n}\left( \omega_{n,m}, \pi \right) \neq 0$ et, si $\pi$ est le quotient de Langlands (\textit{cf.} section \ref{lademo}) de $\tau_1 \times \dots \times \tau_N$, o\`u $\tau_1,\dots ,\tau_N$ sont des repr\'esentations essentiellement de carr\'e int\'egrable, alors  $\pi'$ est le quotient de Langlands de $$\nu ^{-\frac{m-n-1}{2}} \times \dots \times \nu ^{\frac{m-n-1}{2}} \times \widetilde{\tau_1} \times \dots \times \widetilde{\tau_N},$$
o\`u, pour toute repr\'esentation $\tau$, $ \widetilde{\tau}$ d\'esigne sa contragrediente.
\end{enumerate}
\end{theo}
Ainsi, si l'on note $\pi^\ast$ les param\`etres galoisiens de Langlands de la repr\'esentation $\pi$, les param\`etres de $\theta(\pi)$ sont $\widetilde{\pi}^\ast \oplus 1_{m-n}^\ast$ o\`u on a not\'e $1_{m-n}^\ast$ les param\`etres galoisiens de la repr\'esentation triviale de $G_{m-n}$.

La preuve du th\'eor\`eme \ref{theo} se d\'ecompose en trois parties. D'abord, la th\'eorie des fonctions z\^eta de Godement-Jacquet \cite{GJ} nous fournit un entrelacement entre $\omega_{n,n}$ et $\pi \otimes \widetilde{\pi}$ pour toute repr\'esentation irr\'eductible $\pi$, ce qui implique, avec un argument classique (\textit{cf.} \cite[3.II.5]{MVW}), que, si $n \leq m$, alors  $\Hom_{G_n}\left( \omega_{n,m}, \pi \right) \neq 0$.

Pour montrer l'unicit\'e de la repr\'esentation $\theta(\pi)$, on a besoin d'utiliser l'article \cite{Min1} o\`u il est prouv\'e que l'induite parabolique d'une repr\'esentation irr\'eductible a, dans beaucoup de cas, une seule sous-repr\'esentation irr\'eductible. 

Dans la section \ref{filtraciones1}, on d\'ecrit explicitement le \textit{bord} de la repr\'esentation m\'etaplectique (un sous-ensemble de repr\'esentations de $G_n$) et on trouve que, pour toute repr\'esentation qui n'appara\^it pas dans ce bord, la repr\'esentation $\theta(\pi)$ est unique. 

Apr\`es, dans la section \ref{filtraciones}, on s'inspire de l'article \cite{kudla}, et on calcule une filtration des foncteurs de Jacquet de la repr\'esentation m\'etaplectique. Ceci nous permet de montrer dans les sections \ref{application} et \ref{unic}, par r\'ecurrence, l'unicit\'e de la repr\'esentation $\theta(\pi)$, pour les \textit{bonnes} repr\'esentations $\pi$. Les \textit{mauvaises} repr\'esentations sont celles qui ont un foncteur de Jacquet bien pr\'ecis. Or, ces repr\'esentations n'apparaissent pas dans le bord de la repr\'esentation m\'etaplectique!

Pour montrer la param\'etrisation de la correspondance on a, \`a nouveau, deux cas. Par r\'ecurrence, le cas des \textit{bonnes} repr\'esentations n'est pas tr\`es difficile et d\'ecoule de \cite[Corollaire A.3]{Min1}. Pour les autres, on utilise, dans la section \ref{fin}, des propri\'et\'es subtiles de la classification de Zelevinsky des repr\'esentations irr\'eductibles en termes de segments.

Je voudrais particuli\`erement remercier Colette M\oe glin qui m'a prodigu\'e nombre de conseils et id\'ees ainsi que Guy Henniart pour toutes ses suggestions et critiques. Je remercie  aussi Goran Muic et Vincent S\'echerre pour les remarques et corrections qu'ils m'ont faites \`a propos de cet article.

\section{Pr\'eliminaires}
Soient $F$ un corps commutatif localement compact non archim\'{e}dien de caract\'{e}ristique
r\'{e}siduelle $p>0$, $D$ une alg\`{e}bre \`{a} division de centre $F$ et de dimension finie $d^2$ sur $F$.

Soient $n,m$ deux entiers strictement positifs. On note $\mathcal{M}_{n,m}$ (resp. $\mathcal{M}_{n}$) l'ensemble des matrices $n \times m$ (resp. $n \times n$) \`a coefficients dans $D$ et $\ry :\mathcal{M}_n\rightarrow F$ la norme r\'{e}duite. Le groupe ${\rm GL}_n(D)$ des matrices inversibles dans $\mathcal{M}_{n}$ sera not\'e $G_n$. Le groupe trivial sera not\'e $G_0$.

A toute partition (ordonn\'ee) $\alpha=\left( n_1,\dots,n_r\right)$ de l'entier $n$, correspond une d\'ecomposition en blocs des matrices carr\'ees d'ordre $n$. On notera $M_{\alpha}$ le sous-groupe de $G_n$ form\'e des matrices inversibles diagonales par blocs, $P_{\alpha}$ (resp. $\overline{P}_{\alpha}$) le sous-groupe form\'e des matrices triangulaires sup\'erieures (resp. inf\'erieures) par blocs, et $U_{\alpha}$ le sous-groupe de $P_{\alpha}$ form\'e des \'el\'ements dont les blocs diagonaux sont des matrices unit\'e. Le sous-groupe $\overline{P}_{\alpha}$ est conjugu\'e \`a $P_{\overline{\alpha}}$ dans $G_n$ avec $\overline{\alpha}=\left( n_r,\dots,n_1\right)$. 

Dans cet article on ne consid\'erera que des repr\'esentations lisses complexes et le mot \textit{repr\'esentation} voudra toujours dire \textit{repr\'esentation lisse complexe}. On notera $\Irr (G_n)$ l'ensemble des classes d'\'equivalence des repr\'esentations irr\'eductibles de $G_n$. La repr\'esentation triviale de $G_n$ sera not\'ee $1_n$.

Si $\pi$ et $\pi'$ sont deux repr\'esentations d'un groupe $G$, on notera
$$\Hom_G \left(\pi, \pi' \right)$$
l'espace des entrelacements entre $\pi$ et $\pi'$. On omettra l'indice $G$ quand il n'y a pas de confusion.

On note $\sharp\!-\!r_{n_1,\dots,n_r}^{G_n}$ (resp. $\sharp\!-\! \overline{r}_{n_1,\dots,n_r}^{G_n}$) le foncteur de Jacquet non normalis\'e associ\'e au parabolique standard $P_{\alpha}$ (resp. $\overline{P}_{\alpha}$).  On note 
\begin{eqnarray*}
&&r_{n_1,\dots,n_r}^{G_n}=\delta_{P_\alpha}^{-1/2} \sharp\!-\!r_{n_1,\dots,n_r}^{G_n},\\
&\text{(resp.}& \overline{r}_{n_1,\dots,n_r}^{G_n}=\delta_{\overline{P}_\alpha}^{-1/2} \sharp\!-\!\overline{r}_{n_1,\dots,n_r}^{G_n} \text{ ),}
\end{eqnarray*}
le foncteur de Jacquet normalis\'e.

Etant donn\'ee une repr\'esentation $\rho_i$ de chaque $G_{n_i}$, on notera $$\sharp\!-\!\ind^{G_n}_{P_{\alpha}}\left(\rho_1 \otimes \dots \otimes \rho_r\right)$$ l'induite parabolique non normalis\'ee, o\`u on a prolong\'e la repr\'esentation $\rho_1 \otimes \dots \otimes \rho_r$ trivialement sur $U_\alpha$.

On note aussi $\rho_1\times \dots \times \rho_r$ la repr\'esentation $$\ind^{G_n}_{P_{\alpha}}\left(\rho_1 \otimes \dots \otimes \rho_r\right)=\delta_{P_\alpha}^{1/2}\sharp\!-\!\ind^{G_n}_{P_{\alpha}}\left(\rho_1 \otimes \dots \otimes \rho_r\right),$$ induite parabolique normalis\'ee.

Soit $\pi$ une repr\'esentation de $G_n$; on a un isomorphisme canonique (r\'eciprocit\'e de Frobenius): 
\begin{equation}\label{frob}
\Hom \left( \pi, \rho_1\times \dots \times \rho_r \right) \simeq \Hom \left( r_{n_1,\dots,n_r}^{G_n}(\pi), \rho_1 \otimes \dots \otimes \rho_r \right).
\end{equation}

On a une formule similaire pour l'induction non normalis\'ee et le foncteur de Jacquet non normalis\'e.
On dispose aussi d'un isomorphisme de r\'eciprocit\'e \textit{\`a la Casselman} (\textit{cf.} \cite[Theorem 20]{ber}):
\begin{equation}\label{frobcas}
\Hom \left( \rho_1\times \dots \times \rho_r , \pi\right) \simeq \Hom \left( \rho_1 \otimes \dots \otimes \rho_r ,\overline{r}_{n_1,\dots,n_r}^{G_n}(\pi) \right).
\end{equation}
Pour l'induction non normalis\'ee et le foncteur de Jacquet non normalis\'e, la formule pr\'ec\'edente devient:
\begin{eqnarray}\label{frobcas2}
&&\Hom \left( \sharp\!-\!\ind^{G_n}_{P_{\alpha}}\left(\rho_1 \otimes \dots \otimes \rho_r\right) , \pi\right) \simeq  \notag\\ && \hspace{4cm} \Hom \left( \rho_1 \otimes \dots \otimes \rho_r ,\delta_{P_\alpha}\sharp\!-\!\ \overline{r}_{n_1,\dots,n_r}^{G_n}(\pi) \right).
\end{eqnarray}

Soient $n,t \in \mathbb{Z}$, $1 \leq t \leq n$, $\pi \in \Irr(G_n)$, $\chi \in \Irr(G_t)$. On notera $\Jac_\chi(\pi) \neq 0$ (resp. $\overline{\Jac}_\chi (\pi) \neq 0$) s'il existe $\rho \in \Irr(G_{n-t})$ tel que $\Hom\left( \pi, \chi \times \rho \right) \neq 0$ (resp. $\Hom\left( \pi, \rho \times \chi \right) \neq 0$).

On utilisera \`a plusieurs reprises la proposition suivante \cite[Proposition 7.1]{Min1}:
\begin{proposition}
Soient $\pi, \pi' \in \Irr, \rho \in \mathcal{C}$. Les conditions suivantes sont \'equivalentes:\begin{enumerate}
\item $\Hom \left( \pi' , \pi \times \rho \right) \neq 0;$
\item $\Hom \left( \rho \times \pi, \pi' \right) \neq 0. $
\end{enumerate}
\end{proposition}

\begin{corollary}\label{cambio}
Soient $\pi, \pi' $ deux repr\'esentations irr\'eductibles et  $\rho$ une repr\'esentation cuspidale. Les conditions suivantes sont \'equivalentes:\begin{enumerate}
\item $\Hom \left( \pi' , \pi \times \rho \times \dots \times \rho \right) \neq 0;$
\item $\Hom \left( \rho \times \dots \times \rho \times \pi, \pi' \right) \neq 0. $
\end{enumerate}
\end{corollary}
\begin{proof}
Puisque, par \cite[Th\'eor\`eme 5.1]{Min1}, $\pi \times \rho \times \dots \times \rho$ n'a qu'un seul sous-module irr\'eductible et que, par \cite[Th\'eor\`eme 5.6]{Min1},  $ \rho \times \dots \times \rho \times \pi$ n'a qu'un seul quotient irr\'eductible, le corollaire d\'ecoule de la proposition pr\'ec\'edente par r\'ecurrence.
\end{proof}
Si $X$ est un espace localement profini, on note $S(X)$ l'espace vectoriel des fonctions $ \Phi :X \rightarrow \mathbb{C}$ localement constantes \`{a} support compact.
Le lemme suivant sera utilis\'e dans le calcul explicite de la correspondance:
\begin{lemma}\label{alfin}
Soit $X$ un espace localement profini, $X'$ un sous-espace ferm\'e de $X$. Supposons qu'un groupe  localement profini $G$ agit de fa\c{c}on continue sur $X$ et que $G \cdot X'=X$. Notons $H$ le stabilisateur de $X'$ dans $G$. Notons aussi $\pi$ la repr\'esentation naturelle (\textit{cf.} \cite[\textsection 1.2.2]{BZ1}) de $G$ dans $S(X)$ et $\rho$ la repr\'esentation naturelle de $H$ dans $S(X')$. Alors:
$$\pi \simeq \sharp\!-\!\ind^G_H\left( \rho \right).$$
\end{lemma}
\begin{proof}
Posons
\begin{eqnarray*}
\Xi: \pi & \rightarrow & \sharp\!-\!\ind^G_H\left( \rho \right) \\
\phi & \mapsto & \left( g\mapsto \left(\pi(g)\phi\right) |_{X'}  \right).
\end{eqnarray*}
$\Xi$ est bien d\'efini car $X'$ est ferm\'e dans $X$ (\textit{cf.} \cite[\textsection 1.1.8]{BZ1}) et c'est un entrelacement entre $\pi$ et $\sharp\!-\!\ind^G_H\left( \rho \right)$.

Construisons une inverse:
Soit $f \in \sharp\!-\!\ind^G_H\left( \rho \right)$. On d\'efinit $\phi_f \in S(X)$ par
$$\phi_f(x)= f(g)(x'),$$
si $x= g \cdot x'$ et $g\in G$ et $x' \in X'$. Puisque $G \cdot X'=X$ des tels couples $(g, x')$ existent et si $g_1, g_2 \in G$ et $x'_1, x'_2 \in X'$ sont tels que $x = g_1 \cdot x'_1= g_2 \cdot x'_2$, alors $x'_1= g_1^{-1} g_2 \cdot x'_2$ et donc, si l'on pose $h= g_1^{-1} g_2$ on a que $h\in H$.

Ainsi
\begin{eqnarray*}
 f(g_1)(x'_1)&=& f(g_1)(g_1^{-1} g_2 \cdot x'_2) \\
 &=& \rho(h) f(g_1) (x'_2) \\
 &=& f(g_1h)(x'_2) \\
 &=& f(g_2)(x'_2).
\end{eqnarray*}
Donc $\phi_f$ est bien d\'efinie et le morphisme $f \mapsto \phi_f$ est un entrelacement entre $\sharp\!-\!\ind^G_H\left( \rho \right)$ et $\pi$, inverse de $\Xi$.
\end{proof}

On note $S_{n,m}=S\left( \mathcal{M}_{n,m}\right)$. La repr\'{e}sentation m\'{e}taplectique $\omega_{n,m}$
restreinte \`{a} la paire duale $G_{n}\times G_{m}$ (\textit{cf.} Introduction) est la repr\'esentation
$$\omega_{n,m}(g,g')=\nu(g)^{\frac{-m}{2}}\sigma_{n,m}(g,g')\nu(g')^{\frac{n}{2}}.$$

Il est plus naturel de travailler avec la repr\'esentation m\'etaplectique \textit{tordue} $\sigma_{n,m}$ et de calculer ses quotients irr\'eductibles. On d\'eduira ensuite imm\'ediatement les r\'esultats pour la repr\'esentation $\omega_{n,m}$.

Il est aussi tr\`es pratique d'utiliser la notation suivante: on a deux groupes lin\'eaires agissant, par multiplication, sur un espace de matrices \`a gauche et \`a droite. Dor\'enavant, pour diff\'erentier ces deux actions, on notera $G'$, $P'$ et $U'$ les groupes lin\'eaire, parabolique et unipotent respectivement, agissant \`a droite et on gardera les notations $G$, $P$ et $U$ pour ces groupes quand ils agissent \`a gauche. De m\^eme, en cas d'ambigu\"it\'e, on notera $\nu'$ le caract\`ere $\nu$ quand il agit sur $G'$. Il peut sembler une notation un peu artificielle mais elle facilite \'enorm\'ement la compr\'ehension des calculs.

On permet les cas $m=0$ ou $n=0$ (avec $G_0=0$ ou $G'_0=0$) pour lesquels $\mathcal{M}=0$ et $\sigma _{0,m}\simeq \mathbb{C}$ est la repr\'esentation triviale de $G'_m$ et $\sigma _{n,0}\simeq \mathbb{C}$ est la repr\'esentation triviale de $G_n$. 

\section{Le bord de la repr\'esentation m\'etaplectique}\label{filtraciones1}
Dans cette section, on rappelle les r\'esultats de \cite[3.III]{MVW} et on en d\'eduit quelques premi\`eres cons\'equences. On fixe des entiers positifs $n$ et $m$.

Commen\c{c}ons par rappeler que la th\'eorie des fonctions z\^eta de Gode\-ment-Jacquet \cite{GJ} nous fournit, pour toute repr\'esentation irr\'eductible $\pi$ de $G_n$, un entrelacement (\textit{cf.} \cite[3.II.7]{MVW}) entre $\sigma_{n,n}$ et $\pi \otimes \widetilde{\pi}$. Par \cite[3.II.5]{MVW}, on d\'eduit que,  pour toute repr\'esentation irr\'eductible $\pi$ de $G_n$, il existe un sous-quotient irr\'eductible $\pi'$ de $\sharp\!-\!\ind^{G'_m}_{ P'_{m-n,n}}( 1_{m-n} \otimes  \widetilde{\pi})$ tel que
\begin{equation}\label{una}
\Hom_{G_n\times G'_m}\left( \sigma_{n,m}, \pi\otimes\pi'\right) \neq 0.\end{equation}

Ainsi, si $n \leq m$, alors  $\Theta(\pi) \neq 0$. Le probl\`eme est de montrer que ce sous-quotient $\pi'$ est l'unique satisfaisant \`a \eqref{una} et de d\'eterminer ses param\`etres.

D'un autre c\^ot\'e, la repr\'{e}sentation $\sigma _{n,m}$ admet (\textit{cf.} \cite[3.II.2]{MVW}) une filtration 
\begin{equation*}
0=S_{t+1}\subset S_{t}\subset \dots\subset S_{1}\subset S_{0}=S_{n,m},
\end{equation*}
o\`{u} $S_{k}$ est le sous-espace vectoriel de $S_{n,m}$ form\'e des fonctions dont le support est form\'e des matrices de rang \footnote{On utilise la d\'efinition de rang sur une alg\`ebre \`a division de \cite[\textsection 10.12]{Bou}}  plus grand ou \'egal \`a $k$, $0\leq k\leq t=\min \left( n,m\right)$.  L'espace $S_{k+1}$ est ouvert dans $S_k$ par \cite[\textsection 1.1.8]{BZ1} et, en appliquant le lemme \ref{alfin} avec $X'=\left( 
\begin{array}{ll}
0 & 0 \\ 
0 & 1_i
\end{array}
\right)$, on montre dans \cite[3.II.2]{MVW} qu'on a un isomorphisme
$$\sigma _{k}= S_k/S_{k+1} \simeq \sharp\!-\!\ind^{G_{n}G'_m}_{\overline{P}_{n-k,k}P'_{m-k,k}}\left(\mu
_{k}\right), $$ o\`{u} $\mu _{k}$ est la repr\'{e}sentation de $%
\overline{P}_{n-k,k}P'_{m-k,k}$ sur $S\left( G_{k}\right) $ d\'{e}finie par: 
\begin{equation*}
\mu _{k}\left( p,p^{\prime }\right) \Phi \left( h\right) =\Phi \left(
p_{4}^{-1}hp_{4}^{\prime }\right) =\rho _{k}\left( p_{4},p_{4}^{\prime
}\right) \Phi \left( h\right),
\end{equation*}
pour $\Phi \in S\left( G_{k}\right) ,$ $h\in G_{k}$, $p=\left( 
\begin{array}{ll}
p_{1} & 0 \\ 
p_3 & p_{4}
\end{array}
\right) ,$ $p^{\prime }=\left( 
\begin{array}{ll}
p_{1}^{\prime } & p_{2}^{\prime } \\ 
0 & p_{4}^{\prime }
\end{array}
\right) $ et $\rho_k$ la repr\'esentation naturelle de $ G_{k} \times  G'_{k}$ sur  $S\left( G_{k}\right)$ d\'efinie par
\begin{equation*}
\rho _{k}\left( p_{4},p_{4}^{\prime
}\right) \Phi \left( h\right)=
\Phi \left(
p_{4}^{-1}hp_{4}^{\prime } \right).\end{equation*}
\begin{defi}
On dit que $\pi \in \Irr(G_n)$ appara\^it dans le bord de la repr\'esentation $\sigma_{n,m}$ s'il existe $k < n$ tel que $\Hom_{G_n}\left(\sigma_k, \pi \right) \neq 0$.
\end{defi}

\begin{lemma}\label{primero}
Soient $\pi \in \Irr (G_n)$, $\pi' \in \Irr (G'_m)$ telles que $\Hom (\sigma_k,\pi\otimes \pi') \neq 0$. Alors il existe $\tau, \tau' \in \Irr (G_k)$ telles que 
\begin{eqnarray*}\Hom \left( \sharp\!-\!\ind^{G_{n}}_{\overline{P}_{n-kk}}\!\! \left( 1_{n-k} \otimes\tau\right)\otimes \sharp\!-\!\ind^{G'_m}_{P'_{m-k,k}}\!\! \left( 1_{m-k} \otimes \tau'\right), \pi \otimes \pi' \right) \neq 0.
\end{eqnarray*}
\end{lemma}
\begin{proof}
Soient $\pi \in \Irr (G_n)$, $\pi' \in \Irr (G'_m)$ telles que $\Hom (\sigma_k,\pi\otimes \pi') \neq 0$. On a donc un entrelacement non nul de  
$\sharp\!-\!\ind^{G_{n}G'_m}_{\overline{P}_{n-k,k}P'_{m-k,k}}\left( \mu
_{k}\right)$ dans  $\pi \otimes \pi' $. Ceci \'equivaut, par \eqref{frobcas2}, \`a l'existence d'un entrelacement non nul entre
$\mu
_{k}$ et $\delta_{P_{n-k,k}}\sharp\!-\!r_{n-k,k}^{G_n}( \pi) \otimes \delta_{P'_{m-k,k}}\sharp\!-\!\overline{r}_{m-k,k}^{G'_m}( \pi').$ 

Soit $V $ l'image d'un tel entrelacement; il existe une sous-repr\'esentation irr\'eductible $V'$ de $V$ et toute telle sous-repr\'esentation est de la forme $\left( 1_{n-k} \otimes\tau\right)\otimes \left( 1_{m-k} \otimes \tau'\right)$ comme repr\'esentation de $\overline{M}_{n-k,k}M'_{m-k,k}$ o\`u $\tau$ et $\tau'$ sont irr\'eductibles. Par \eqref{frobcas2}, \`a nouveau, on trouve le r\'esultat.
\end{proof}
\begin{corollary}\label{condit}
Soit $\pi \in \Irr (G_n)$. Les conditions suivantes sont \'equivalentes:
\begin{enumerate}
\item La repr\'esentation $\pi$ n'appara\^it pas dans le bord de $\sigma_{n,m}$. 
\item Il n'existe pas $\tau \in \Irr (G_k)$, $k<n$, telle que $$\Hom_{G_n} \left( \sharp\!-\!\ind^{G_{n}}_{\overline{P}_{n-k,k}}\left( 1_{n-k} \otimes \tau\right), \pi \right) \neq 0.$$
\end{enumerate}
\end{corollary}
\begin{proof}
Pour l'implication directe, supposons qu'il existe $\tau \in \Irr (G_k)$, $k<n$, et un entrelacement non nul de $\sharp\!-\!\ind^{G_{n}}_{\overline{P}_{n-k,k}}\left( 1_{n-k} \otimes \tau\right)$ dans $ \pi$. Par exactitude du foncteur $\sharp\!-\!\ind$ et \cite[Lemme 3.II.3]{MVW}, on a un morphisme surjectif de $\sigma_k$ dans $\sharp\!-\!\ind^{G_{n}}_{\overline{P}_{n-k,k}}\left( 1_{n-k} \otimes \tau\right)$ qui compos\'e avec l'entrelacement pr\'ec\'edent nous montre que $\Hom_{G_n}\left(\sigma_k, \pi \right) \neq 0$, \textit{i.e} que la repr\'esentation $\pi$ appara\^it dans le bord de $\sigma_{n,m}$.
\end{proof}

Rappelons que, d'apr\`es \cite[Th\'eor\`eme 5.1]{Min1}, pour toute repr\'esentation irr\'eductible $\pi \in \Irr (G_n)$, la repr\'esentation $\sharp\!-\!\ind^{G'_m}_{ P'_{m-n,n}}(1_{m-n} \otimes \pi)$ a un unique quotient irr\'eductible et qu'il appara\^it avec multiplicit\'e $1$ dans l'induite. Le th\'eor\`eme suivant r\'esume tout ce que l'on peut conclure \`a partir de cette filtration par le rang. 
\begin{theorem}\label{primerhowe}
Soient $n,m$ des entiers positifs et supposons $n \leq m$. Soient $\pi \in \Irr (G_n)$ qui n'apparaisse pas dans le bord de $\sigma_{n,m}$. Il existe une unique repr\'esentation $\pi' \in Irr(G'_m)$ telle que $$\Hom_{G_n\times G'_m}\left( \sigma_{n,m}, \pi\otimes\pi'\right)\neq0.$$ C'est l'unique quotient irr\'eductible de $\sharp\!-\!\ind^{G'_m}_{ P'_{m-n,n}}(1_{m-n} \otimes \widetilde{\pi}).$
De plus, $$\dim \left(\Hom_{G_n\times G'_m}\left( \sigma_{n,m}, \pi\otimes\pi'\right) \right)=1.$$
\end{theorem}
\begin{proof}
On avait vu au d\'ebut de la section qu'il existe $\pi' \in Irr(G'_m)$ telle que $\Hom_{G_n\times G'_m}\left( \sigma_{n,m}, \pi\otimes\pi'\right)\neq0$. Montrons qu'elle est unique.

Par composition avec les morphismes $S_j \rightarrow \sigma_{n,m}$, $j=0, \dots , n$, on obtient des morphismes $S_j \rightarrow \pi \otimes \pi'$.

Soit $k$ le plus grand $j$ tel que $S_j \rightarrow \pi \otimes \pi'$ ne soit pas le morphisme nul. On a donc que $\pi \otimes \pi'$ est un quotient de $\sigma_k$. Par hypoth\`ese on a $k=n$ et donc, 
$$\Hom \left( \sigma_n , \pi \otimes \pi' \right) \neq 0.$$
Par d\'efinition de $\sigma_n$ on a alors
$$\Hom \left( \sharp\!-\!\ind ^{G_n \times G'_m}_{G_{n} \times P'_{m-n,n}}\left( \mu_{n}\right)  , \pi \otimes \pi' \right) \neq 0.$$
Par \eqref{frobcas2} et la d\'efinition de $\mu_n$, on d\'eduit que
$$\Hom \left(1_{m-n} \otimes \rho_{n}  , \pi \otimes \delta_{P_{m-n,n}}\sharp\!-\!r_{m-n,n}^{G_m}(\pi') \right) \neq 0,$$
d'o\`u, par \cite[Lemme 3.II.3]{MVW},
$$\Hom \left(1_{m-n} \otimes \widetilde{\pi}  , \delta_{P_{m-n,n}}\sharp\!-\!r_{m-n,n}^{G_m}(\pi') \right) \neq 0,$$
et, \`a nouveau par \eqref{frobcas2}, on d\'eduit que $\pi'$ est l'unique quotient irr\'eductible de $\sharp\!-\!\ind^{G'_m}_{ P'_{m-n,n}}( 1_{m-n} \otimes \widetilde{\pi}).$

Montrons finalement que $$\dim \left( \Hom_{G_n\times G'_m}\left( \sigma_{n,m}, \pi\otimes\pi' \right) \right)=1.$$
Soit $\lambda \in \Hom_{G_n\times G'_m}\left( \sigma_{n,m}, \pi\otimes\pi' \right) $. La compos\'ee de $\lambda$ avec l'inclusion $\sigma_n \hookrightarrow \sigma _{n,m}$ n'est pas nulle. Or, par le lemme \cite[Lemme 3.II.3]{MVW}, $\dim \left( \Hom_{G_n\times G'_m}\left( \sigma_{n}, \pi\otimes\pi' \right) \right)=1,$ et donc cette compos\'ee est unique \`a homoth\'ethie pr\`es. Ainsi, 
si $$\dim \left( \Hom_{G_n\times G'_m}\left( \sigma_{n,m}, \pi\otimes\pi' \right) \right)>1,$$ on pourrait construire, par combinaison lin\'eaire, un morphisme non nul $\lambda \in \Hom_{G_n\times G'_m}\left( \sigma_{n,m}, \pi\otimes\pi'\right)$ tel que sa compos\'ee avec  l'inclusion $\sigma_n \hookrightarrow \sigma _{n,m}$ soit nulle. Il existe alors $k<n$ tel que $\Hom_{G_n} \left( \sigma_k, \pi \right) \neq 0,$ ce qui est absurde, par hypoth\`ese, et qui ach\`eve la d\'emonstration.
\end{proof}

\begin{remark}
En particulier, si $\pi$ est une repr\'esentation cuspidale, il existe une unique $\pi'$ (l'unique quotient irr\'eductible de $\sharp\!-\!\ind^{G'_m}_{ P'_{m-n,n}}(1_{m-n} \otimes \widetilde{\pi})$), telle que
$$\Hom_{G_n\times G'_m}\left( \sigma_{n,m}, \pi\otimes\pi'\right)\neq0$$
En effet, si $\pi$ est cuspidale, il n'existe pas $\tau \in \Irr (G_k)$, $k<n$, telle que $$\Hom_{G_n} \left( \sharp\!-\!\ind^{G_{n}}_{\overline{P}_{n-k;k}}\left( 1_{n-k} \otimes \tau\right), \pi \right) \neq 0,$$
et donc, par le corollaire \ref{condit}, elle n'appara\^it pas dans le bord de $\sigma_{n,m}$.
\end{remark}

\section{Suite de Kudla}\label{filtraciones}
Dans cette section, on refait les calculs de \cite{kudla}, pour les paires duales de type II. On veut calculer une suite de composition des foncteurs de Jacquet de la repr\'esentation $\sigma_{n,m}$.

Soient $t,i$ des entiers, $0<t \leq n$, $0 \leq i \leq \inf \left\{t,m \right\}$. Fixons quelques notations pour cette section:

\begin{itemize}
\item On notera chaque matrice $\mathbf{m} \in \mathcal{M}_{n,m}$ par 
\begin{equation*}
\mathbf{m}=\left( 
\begin{array}{l}
a \\ b

\end{array}
\right) =
\left( 
\begin{array}{c}
a \\
\begin{array}{ll}
m_{1} &
m_{2}
\end{array}

\end{array}
\right) ,
\end{equation*}
o\`{u} $a \in \mathcal{M}_{t,m},b\in \mathcal{M}_{n-t,m}, m_{1}\in \mathcal{M}_{n-t,i},\ m_{2}\in \mathcal{M}_{n-t,m-i}, .$

\item On notera chaque $g\in M_{(t,n-t)}$ par 
\begin{equation*}
g=\left( 
\begin{array}{ll}
g_{0} & \\ 
 & g_{1}
\end{array}
\right)
=\left( 
\begin{array}{ll}
 
\begin{array}{ll}
g_{2} & g_{3} \\ 
g_{4} & g_{5}
\end{array}
 
 & \\ 
 & g_{1}
\end{array}
\right), 
\end{equation*}
$g_{0}\in G_{t}, \, g_{1}\in G_{n-t},\ g_{2}\in \mathcal{M}_{t-i,t-i},\, g_{3}\in \mathcal{M}_{t-i,i},\, g_{4}\in \mathcal{M}_{i,t-i},\, g_{5}\in \mathcal{M}_{i,i}.$

\item On notera chaque $g^{\prime }\in G'_{m}$ par 
\begin{equation*}
g^{\prime }=\left( 
\begin{array}{ll}
g_{1}^{\prime } & g_{2}^{\prime } \\ 
g_{3}^{\prime } & g_4^{ \prime }
\end{array}
\right) ,
\end{equation*}

o\`{u} $ g_{1}^{\prime } \in \mathcal{M}_{i,i}, \,g_{2}^{\prime }\in \mathcal{M}_{m-i,i}, \,g_{3}^{\prime }\in \mathcal{M}_{i,m-i},\,g_{4}^{\prime }\in \mathcal{M}_{m-i,m-i}.$
\end{itemize}

Soit $\psi$ un caract\`ere non trivial de $F$.

On d\'efinit $\sigma_{n,m} ^{\ast }$ par le diagramme commutatif suivant: 
\begin{equation}\label{fourier}
\xymatrix { S_{n,m} \ar[r]^{\sigma_{n,m} \left(
g,g^{\prime }\right)} \ar[d]_{\widehat{}}& S_{n,m} \ar[d]_{\widehat{}}\\ S_{n,m} \ar[r]^{\sigma_{n,m}^{*} \left( g,g^{\prime }\right)} &
S_{n,m}},
\end{equation}
o\`u $\,\widehat{}\,$ est l'isomorphisme de repr\'esentations qui envoie $f\in S_{n,m} $ vers
\begin{equation*}
\widehat{f}
\left( 
\begin{array}{l}
a \\ b
\end{array}
\right)=\int_{\mathcal{M}_{t,m}}f\left( 
\begin{array}{l}
a^{\ast} \\ b

\end{array}
\right)\psi \circ \trd\left( ^{t}a a^{\ast } \right) da^{\ast},
\end{equation*}
o\`u on a not\'e $\trd$ la trace r\'eduite.

On se propose d'\'{e}tudier la repr\'{e}sentation $\sigma_{n,m} ^{\ast }$
(transform\'{e}e de Fourier partielle de $\sigma_{n,m} $, isomorphe \`a $\sigma_{n,m}$ par \eqref{fourier}). Plus tard, on ne fera pas de distinction entre les deux repr\'esentations.

La repr\'esentation $\sigma_{n,m} ^{\ast }$ agit sur $\widehat{f}$ par 
\begin{equation}\label{action}
\sigma_{n,m} ^{\ast }\left( g,g^{\prime }\right) \widehat{f}\left( 
\begin{array}{l}
a \\ b
\end{array}
\right) =\int_{\mathcal{M}_{n,t}}f\left( g^{-1}\left( 
\begin{array}{l}
a^{\ast} \\ b
\end{array}
\right) g^{\prime }\right) \psi \circ \trd\left( ^{t}aa^{\ast }\right)
da^{\ast }.
\end{equation}
Ainsi $U_{t,n-t}$ agit, via $\sigma_{n,m} ^{\ast }$, par 
\begin{eqnarray}\label{U=0}
\sigma_{n,m} ^{\ast }\left( \begin{array}{ll}
1 & u \\
& 1
\end{array}\right)\widehat{f} \left( 
\begin{array}{l}
a \\ b
\end{array}
\right) &=&\psi \circ \trd\left( ^{t}aub\right)\widehat{f} \left( 
\begin{array}{l}
a \\ b
\end{array}
\right) \notag\\
&=& \psi \circ \trd\left( b^{t}au\right)\widehat{f} \left( 
\begin{array}{l}
a \\ b
\end{array}
\right)
\end{eqnarray}

Notons $A$ la partie ferm\'ee de $\mathcal{M}_{n,m}$:
$$A= \left\{ \left( 
\begin{array}{l}
a \\ b
\end{array}
\right) \in \mathcal{M}_{n,m}: b^{t}a=0 \right\} .$$
\begin{lemma}
Le sous-espace $S_{n,m}\left(U_{t,n-t}, \sigma_{n,m} ^{\ast }\right)$ de $S_{n,m}$ engendr\'e par les fonctions de la forme
$$\widehat{f}-\sigma_{n,m} ^{\ast }(u)\widehat{f},$$
o\`u $f \in S_{n,m}$ et $u \in U_{t,n-t}$ est l'espace des $\widehat{f}\in S_{n,m}$ telles que
\begin{equation*}
\supp\widehat{f} \cap A =\emptyset ,
\end{equation*}
\end{lemma}
\begin{proof}
Sur $A$, $\widehat{f}-\sigma_{n,m} ^{\ast }(u)\widehat{f}$ est, par \eqref{U=0}, nul. R\'eciproquement, soit $\widehat{f}\neq0$ nulle sur $A$ et montrons que $\widehat{f} \in S_{n,m}\left(U_{t,n-t}, \sigma_{n,m} ^{\ast }\right)$. Le lemme sera d\'emontr\'e. 

Pour tout $\mathbf{m}= \left(  \begin{array}{l}
a \\ b
\end{array}
\right)\in \mathcal{M}_{n,m}$, notons $\alpha_\mathbf{m}$ le caract\`ere d\'efini par
\begin{eqnarray*}
\alpha_\mathbf{m} &:& U_{t,n-t} \rightarrow \mathbb{C} \\
&&u \mapsto \psi \circ \trd\left( ^{t}aub\right).
\end{eqnarray*}

Pour tout $\mathbf{m}\in \mathcal{M}_{n,m} \backslash A$, il existe $u_\mathbf{m}\in U_{t,n-t}$ tel que $\alpha_\mathbf{m}(u_\mathbf{m})\neq 1$. Pour $\mathbf{m}'$ dans un voisinage $I$ de $\mathbf{m}$, on a par continuit\'e $\alpha_{\mathbf{m}'}(u_\mathbf{m})=\alpha_\mathbf{m}(u_\mathbf{m}) \neq 1$. Le support de $\widehat{f}$, \'etant compact, il suffit de montrer que la fonction caract\'eristique $\mathbf{1}_I$ du voisinage $I$ est de la forme $\widehat{F}-\sigma_{n,m} ^{\ast }(u)\widehat{F}$, pour $F \in S_{n,m}$.

Puisque, $\alpha_{\mathbf{m}'}(u_\mathbf{m})=\alpha_\mathbf{m}(u_\mathbf{m}) \neq 1$ on a que $\mathbf{1}_I$ est \'egal \`a $\sigma_{n,m} ^{\ast }(u)\mathbf{1}_I$ \`a une constante non nulle pr\`es. Ainsi, $\widehat{F}=\frac{1}{1-\alpha_\mathbf{m}(u_\mathbf{m}) }\mathbf{1}_I$ convient.


\end{proof}
Ainsi la suite exacte courte (\textit{cf.} \cite[\textsection 1.1.8]{BZ1}) 
\begin{eqnarray*}
0 \rightarrow S\left(\mathcal{M} \backslash A\right) \rightarrow S_{n,m} &\rightarrow &S(A) \rightarrow 0 \\
\widehat{f} & \mapsto & \widehat{f} |_A
\end{eqnarray*}
s'identifie \`a
$$0 \rightarrow S_{n,m}\left(U_{t,n-t}, \sigma_{n,m} ^{\ast }\right) \rightarrow S_{n,m} \rightarrow S(A) \rightarrow 0$$
et donc $S(A)$, muni de l'action de $M_{(t,n-t)}$ donn\'ee par \eqref{action}, s'identifie au foncteur de Jacquet non normalis\'e, d'o\`u un isomorphisme de $M_{(t,n-t)}$-modules:
$$S(A) \simeq \sharp\!-\! r_{t,n-t}^{G_n}\left( \sigma_{n,m} \right) .$$

On a alors une filtration de $\sharp\!-\! r_{t,n-t}^{G_n}\left( \sigma_{n,m} \right) $ de $M_{(t,n-t)}\times G'_{m}$-modules:
\begin{equation*}
0=S_{k+1}\subset S_{k}\subset \dots\subset S_{1} \subset S_{0}=\sharp\!-\! r_{t,n-t}^{G_n}\left( \sigma_{n,m} \right) ,
\end{equation*}
o\`{u} 
\begin{eqnarray*}
S_{i} &=&\left\{  \widehat{f} \in S(A) : \widehat{f}\left( 
\begin{array}{l}
a \\ b
\end{array}
\right)=0 \text{ si } \rang\left( a\right)
\leq i-1\right\} \\
k &=&\inf \left( t,m\right) .
\end{eqnarray*}
Chaque $S_{i+1}$ est ouvert dans $S_{i}$.
La repr\'esentation $\sharp\!-\! r_{t,n-t}^{G_n}\left( \sigma_{n,m} \right)$ est donc compos\'ee des repr\'esentations
\begin{eqnarray*}
\tau^\ast _{i} &=&S_{i} / S_{i+1} \qquad  0\leq i\leq k.
\end{eqnarray*}

La suite exacte de $M_{(t,n-t)}\times G'_{m}$-modules (\textit{cf.} \cite[\textsection 1.1.8]{BZ1}) 
$$0 \rightarrow S_{i}  \rightarrow S_{i+1} \rightarrow S(A_i) \rightarrow 0$$
o\`u $A_i=\left\{ \left( 
\begin{array}{l}
a \\ b
\end{array}
\right) \in \mathcal{M}_{n,m}: b^{t}a=0 \text{ et } \rang (a)=i \right\} ,$
nous montre que l'espace de $\tau^{\ast} _{i} $ est l'espace des fonctions $\left\{ \widehat{f} \in S(A_i) \right\}$ et $M_{(t,n-t)}\times G'_{m}$ agit, via $\tau^{\ast} _{i} $, sur cet espace par \eqref{action}.

Maintenant on va appliquer le lemme \ref{alfin}, avec $X=A_i$ et $X'$ l'ensemble des matrices de la forme $\left( 
\begin{array}{c}
a_i \\
\begin{array}{ll}
0 &  \ast
\end{array}
\end{array}
\right)
$, avec $\ast \in \mathcal{M}_{n-t,m-i}$. Le stabilisateur $T_i$ dans $M_{(t,n-t)}\times G'_{m}$ de $X'$ est le sous-groupe des $(g,g') \in P_{t-i,i}\times G_{n-t}\times P'_{i,m-i}$ tels que $g_5^{-1}g'_1=1$. Clairement $M_{(t,n-t)}\times G'_{m} \cdot X'=A_i$. Voyons que l'action de $T_i$ dans $S(X')$ est isomorphe \`a
 $\xi^0
_{t,i}\otimes\sigma _{n-t,m-i}$ o\`u $\xi^0 _{t,i}$
est le caract\`{e}re de $T_i$ d\'efini par $\nu \left( g_0\right) ^m \nu \left( g^{ \prime }\right) ^{-t} $, pour $(g,g') \in T_i.$ 

Soient $(g,g') \in T_i$, $x\in  S_{n-t,m-i}$, et $ \widehat{f} \in S(A_i)$. On a que 

$\tau^{\ast}_i \left( g,g^{\prime }\right)  \widehat{f} \left( 
\begin{array}{c}
a_i \\
\begin{array}{ll}
0 &  x
\end{array}
\end{array}
\right)
 =$

\begin{eqnarray*}
&=&\int_{\mathcal{M}_{t,m}}f\left( g^{-1}
\left( 
\begin{array}{c}
a^{\ast} \\
\begin{array}{ll}
0 &
x
\end{array}

\end{array}
\right) 
g^{\prime }\right) \psi \circ \trd\left( \left( 
\begin{array}{ll}
0 & 1_i \\ 
0 & 0
\end{array}
\right) a^{\ast }\right) da^{\ast }
\\
&=&\hspace{-0.2cm}\int_{\mathcal{M}_{t,m}}\hspace{-0.3cm}f\left( 
\begin{array}{c}
g_0^{-1}a^{\ast} g' \\
 \begin{array}{cc}
0 & g_1^{-1}x g'_3
\end{array}
\end{array}
\right) \psi \circ \trd\left( \left( 
\begin{array}{ll}
0 & 1_i \\ 
0 & 0
\end{array}
\right) a^{\ast }\right) da^{\ast }
\\
&=&\hspace{-0.2cm}\nu \left( g_0\right) ^m \nu \left( g^{ \prime }\right) ^{-t} \hspace{-0.2cm}\int_{\mathcal{M}_{t,m}}\hspace{-0.3cm}f\left( \hspace{-0.1cm}
\begin{array}{c}
a^{\ast} \\
 \begin{array}{cc}
0 & g_1^{-1}x g'_3
\end{array}
\end{array}\hspace{-0.1cm}
\right) \psi \circ \trd\left( \left( 
\begin{array}{ll}
0 & \left( g_{5}^{-1} g'_1 \right)^{-1}\\ 
0 & 0
\end{array}
\hspace{-0.2cm}\right) a^{\ast }\hspace{-0.1cm}\right) da^{\ast }
\\
&=&\hspace{-0.2cm}\nu \left( g_0\right) ^m \nu \left( g^{ \prime }\right) ^{-t} \hspace{-0.2cm}\int_{\mathcal{M}_{t,m}}\hspace{-0.3cm}f\left( \hspace{-0.1cm}
\begin{array}{c}
a^{\ast} \\
 \begin{array}{cc}
0 & g_1^{-1}x g'_3
\end{array}
\end{array}\hspace{-0.1cm}
\right) \psi \circ \trd\left( \left( 
\begin{array}{ll}
0 & 1_i \\ 
0 & 0
\end{array}
\hspace{-0.2cm}\right) a^{\ast }\hspace{-0.1cm}\right) da^{\ast }
\\
&=&\hspace{-0.2cm}\nu \left( g_0\right) ^m \nu \left( g^{ \prime }\right) ^{-t} \sigma_{n-t,m-i}(g_1,g'_3)   \widehat{f} |  _ {\left( 
\begin{array}{c}
a_i \\
\begin{array}{ll}
0 &  \ast
\end{array}

\end{array}
\right)
}\left(x\right).
\end{eqnarray*} 

\begin{proposition}\label{lema1} La repr\'esentation $r_{t,n-t}^{G_n}\left( \sigma_{n,m} \right)$ est compos\'ee des repr\'esentations $\tau_i$, $i=0,\dots, \min \left\{ t,m \right\}$, o\`u  
\begin{equation*}
\tau _{i} \simeq \ind_{P_{t-i,i}\times G_{n-t}\times P'_{i,m-i}}
^{M_{(t,n-t)}\times G'_{m}}\left( \xi
_{t,i}\otimes \rho _{i}\otimes \sigma _{n-t,m-i}\right) ,
\end{equation*}
et o\`u $\xi _{t,i}$
est le caract\`{e}re

$$\xi _{t,i} = \begin{cases} 
\nu ^{\frac{2m-n+t-i}{2}}& \text{sur }G_{t-i} \\
 \nu^{\frac{2m-n+2t-i}{2}}& \text{sur } G_{i} \\
 \nu ^{\frac{t}{2}}& \text{sur } G_{n-t}\\
\nu ^{\frac{-m-2t+i}{2}}& \text{sur } G'_{i} \\
 \nu ^{\frac{-2t+i}{2}}& \text{sur } G'_{m-i}.
  \end{cases}$$
\end{proposition}

\begin{proof}
D'apr\`es le lemme \ref{alfin},  ce qui pr\'ec\`ede implique que
\begin{equation*}
\tau^\ast _{i} \simeq \left(\sharp\!-\!\ind\right)_{T_i}
^{M_{(t,n-t)}\times G'_{m}}\left( \xi^0
_{t,i}\otimes\sigma _{n-t,m-i}\right).
\end{equation*}

Induire de $T_i$ \`a $M_{(t,n-t)}\times G'_{m}$ revient \`a induire de $T_i$ \`a $P_{t-i,i}\times G_{n-t}\times P_{i,m-i}$ et puis \`a $M_{(t,n-t)}\times G'_{m}$.  Or, l'induite $$\left(\sharp\!-\! \ind\right)_{T_i}
^{P_{t-i,i}\times G_{n-t}\times P'_{i,m-i}}\left( \xi^0
_{t,i}\otimes\sigma _{n-t,m-i}\right) $$ est la repr\'esentation $\xi^0
_{t,i}\otimes \rho _{i}\otimes \sigma _{n-t,m-i}$.

Ainsi, l'induite $\left(\sharp\!-\! \ind\right)_{T_i}
^{M_{(t,n-t)}\times G'_{m}}\left( \xi^0
_{t,i}\otimes\sigma _{n-t,m-i}\right) $ est la repr\'esentation $$\left(\sharp\!-\! \ind\right)_{P_{t-i,i}\times G_{n-t}\times P'_{i,m-i}}
^{M_{(t,n-t)}\times G'_{m}}\left( \xi^0
_{t,i}\otimes \rho _{i}\otimes \sigma _{n-t,m-i}\right) .$$
Pour achever la proposition il ne nous reste qu'\`a changer les induites et foncteurs de Jacquet non normalis\'es en induites et foncteurs de Jacquet normalis\'es. Ainsi 
\begin{equation*}
\tau _{i} \simeq \delta_{P_{t,n-t}}^{-\frac{1}{2}}\tau_i^\ast \simeq \ind_{P_{t-i,i}\times G_{n-t}\times P'_{i,m-i}}
^{M_{(t,n-t)}\times G'_{m}}\left( \xi
_{t,i}\otimes \rho _{i}\otimes \sigma _{n-t,m-i}\right) ,
\end{equation*}
avec $$\xi _{t,i}=\nu \left( g_0\right) ^m \nu \left( g^{ \prime }\right) ^{-t} \delta _{P_{t-i,i}}^{\frac{-1}{2}}\delta _{P_{t,n-t}}^{\frac{-1}{2}}\delta _{P'_{i,m-i}}^{\frac{-1}{2}},$$
\textit{i.e.} le caract\`{e}re requis.
\end{proof}

Avec les m\^emes arguments on calcule une suite de composition du foncteur de Jacquet $\overline{r}_{t,m-t}^{G'_m}$, agissant cette fois-ci du c\^ot\'e de $G'_m$. 
\begin{proposition}\label{lema2}
Soient $t,i$ des entiers, $0<t \leq m$ $0 \leq i \leq \inf \left\{t,n \right\}$. 
La repr\'esentation $\overline{r}_{t,m-t}^{G'_m}\left( \sigma_{n,m} \right)$ est compos\'ee des repr\'esentations $\overline{\tau'}_i$, $i=0,\dots, \min \left\{ t,n \right\}$ o\`u
\begin{equation*}
\overline{\tau'} _{i} \simeq \ind_{ P_{n-i,i}\times P'_{i,t-i}\times G'_{m-t}}
^{G_{n}\times M'_{(t,m-t)}}\left( \sigma _{n-i,m-t} \otimes \rho _{i} \otimes \overline{\xi'}
_{t,i}\right) ,
\end{equation*}
et o\`u $\overline{\xi'} _{t,i}$
est le caract\`{e}re 
$$\overline{\xi'} _{t,i}= \begin{cases} 
\nu ^{\frac{2t-i}{2}}& \text{sur }G_{n-i} \\
 \nu^{\frac{n+2t-i}{2}}& \text{sur } G_{i} \\
 \nu ^{\frac{-2n+m-2t+i}{2}}& \text{sur } G'_{i}\\
\nu ^{\frac{m-2n-t+i}{2}}& \text{sur } G'_{t-i} \\
 \nu ^{\frac{-t}{2}}& \text{sur } G'_{m-t}.
  \end{cases}$$
\end{proposition}

\section{Application}\label{application}
Soient $\pi \in \Irr G_n$, $\pi' \in \Irr G'_m$ telles que $\pi \otimes \pi'$ soit un quotient de $\sigma_{n,m}$. Soit $r$ un entier positif. Pour toute repr\'esentation cuspidale $\chi$ de $G_r$ consid\'erons l'eniter positif maximal $a $ tel que $\pi$ soit une sous-repr\'esentation d'une repr\'esentation de la forme
$$ \chi \times \chi \times \dots \times \chi \times \rho,$$ o\`u on a fait le produit de $a$ fois la repr\'esentation $\chi$ fois $\rho$, $\rho$ \'etant une repr\'esentation irr\'eductible de $G_{n-ra}$.
Alors, par exactitude du foncteur de Jacquet, on a un entrelacement surjectif de $r_{ra,n-ra}^{G_n}\left( \sigma_{n,m}\right)$ dans $ r_{ra,n-ra}^{G_n}\left( \pi \right)\otimes \pi' $ d'o\`u un entrelacement non nul de $ r_{ra,n-ra}^{G_n}\left( \sigma_{n,m}\right) $ dans $ \chi \times \chi \times \dots \times \chi \otimes \rho\otimes \pi' .$

D'apr\`es \ref{lema1}, il existe alors $i \in \left\{ 0, \dots , ra \right\}$ tel que
$$\Hom\left( \tau_i , \chi \times \chi \times \dots \times \chi \otimes \rho\otimes \pi' \right)\neq 0. $$

\begin{lemma}\label{lima}
Les seuls $\tau_i$ qui peuvent avoir des quotients de la forme ci-dessus sont $\tau_{ra}$ et $\tau_{ra-1}$ et, si $\chi \neq \nu ^{\frac{2m-n+1}{2}}$ seul $\tau_{ra}$ peut en avoir. 
\end{lemma}
\begin{proof}
En effet, supposons que 
\begin{eqnarray*}
&&\Hom\left( \tau_{i}, \chi \times \chi \times \dots \times \chi \otimes \rho\otimes \pi' \right)\neq 0. \end{eqnarray*}

Cela signifie, par d\'efinition de $\tau_i$, 
\begin{eqnarray*}
&& \Hom \Big{(} \ind_{P_{ra-i,i}\times G_{n-ra}\times P'_{i,m-i}}
^{M_{(ra,n-ra)}\times G'_{m}}\left( \xi
_{ra,i}\otimes \rho _{i}\otimes \sigma _{n-ra,m-i}\right) , \\
&& \qquad \qquad\qquad \qquad\qquad\qquad\chi \times \chi \times \dots \times \chi \otimes \rho\otimes \pi' \Big)\neq 0 \end{eqnarray*}
et, par \eqref{frobcas}, 
\begin{eqnarray*}
&& \Hom\Big( \xi
_{ra,i}\otimes \rho _{i}\otimes \sigma _{n-ra,m-i}, \\
&& \qquad\qquad\overline{r}_{P_{ra-i,i}\times G_{n-ra}\times P'_{i,m-i}}
^{M_{(ra,n-ra)}\times G'_{m}}\left( \chi \times \chi \times \dots \times \chi \otimes \rho\otimes \pi' \right) \Big) \neq 0,
\end{eqnarray*}
d'o\`u
\begin{eqnarray*}
&& \Hom\Big( \xi
_{ra,i}\otimes \rho _{i}\otimes \sigma _{n-ra,m-i}, \\
&& \qquad\qquad\overline{r}_{P_{ra-i,i}}
^{G_{ra}}\left( \chi \times \chi \times \dots \times \chi\right)  \otimes \rho\otimes \overline{r}_{P'_{i,m-i}}
^{G'_{m}}\left(  \pi' \right) \Big) \neq 0,
\end{eqnarray*}
et donc
$$\Hom\left( \xi
_{ra,i}|_{G_{ra-i}}, \overline{r}_{P_{ra-i,i}}^{G_{ra}}
\left( \chi \times \chi \times \dots \times \chi \right)|_{G_{ra-i}} \right) \neq 0,
$$

Ainsi, on trouve finalement
$$\Hom\left(\nu_{ra-i}^{\frac{2m-n+ra-i}{2}} , \chi \times \chi \times \dots \times \chi \right)\neq 0 .$$

Par l'unicit\'e du support cuspidal, il faut alors que $$\supp \left(\nu_{ra-i}^{\frac{2m-n+ra-i}{2}}\right) = \left\{ \chi, \dots, \chi \right\}$$ et donc, ou bien $i=ra$ ou bien $i=ra-1$ et $\chi = \nu ^{\frac{2m-n+1}{2}}$.
\end{proof}
Ainsi on se retrouve avec deux cas:
\begin{enumerate}
\item[{\bf Cas A,}]  $\Hom\left( \tau_{ra}, \chi \times \chi \times \dots \times \chi \otimes \rho\otimes \pi' \right)\neq 0,$
\item[{\bf Cas B,}] $\Hom\left( \tau_{ra-1}, \chi \times \chi \times \dots \times \chi \otimes \rho\otimes \pi' \right)\neq 0$ et $\chi = \nu ^{\frac{2m-n+1}{2}}$.
\end{enumerate}

Examinons successivement les diff\'erents cas du lemme plus haut (le cas B sera trait\'e dans la section \ref{fin}).

\underline{\bf Cas A} Supposons d'abord $\Hom\left( \tau_{ra} , \chi \times \chi \times \dots \times \chi \otimes \rho\otimes \pi' \right)\neq 0 $ (ce qui arrive, en particulier, d'apr\`es le lemme pr\'ec\'edent, pour $\chi$ distinct de $\nu ^{\frac{2m-n+1}{2}}$)

Alors, $\Hom\left( \tau_{ra}, \chi \times \chi \times \dots \times \chi \otimes \rho\otimes \pi' \right)\neq 0 $, s'\'ecrit, d'apr\`es la proposition \ref{lema1}
\begin{eqnarray*}
&& \Hom \Big{(} \ind_{M_{(ra,n-ra)}\times P'_{ra,m-ra}}
^{M_{(ra,n-ra)}\times G'_{m}}\left( \xi
_{ra,ra}\otimes \rho _{ra}\otimes \sigma _{n-ra,m-ra}\right) , \\
&& \qquad \qquad\qquad \qquad\qquad\qquad\chi \times \chi \times \dots \times \chi \otimes \rho\otimes \pi' \Big)\neq 0 \end{eqnarray*}
ce qui implique, par d\'efinition de $ \xi_{ra,ra}$ que
\begin{eqnarray*}
&& \Hom \Big{(} \ind_{P'_{ra,m-ra}}
^{G'_{m}}\left( \rho _{ra}\otimes \sigma _{n-ra,m-ra}\right) , \\
&& \qquad\qquad\qquad\nu ^{\frac{n-m}{2}}\chi \times \nu ^{\frac{n-m}{2}}\chi \times \dots \times \nu ^{\frac{n-m}{2}}\chi \otimes  \nu ^{\frac{-ra}{2}}\rho\otimes  \nu ^{\frac{ra}{2}}\pi' \Big)\neq 0 \end{eqnarray*}
d'o\`u, par le lemme \cite[3.II.3]{MVW},
\begin{eqnarray*}
&\Hom\Big( \ind_{P'_{ra,m-ra}}^{G'_m}\left(\nu ^{\frac{m-n}{2}}\widetilde{\chi} \times \dots \times \nu ^{\frac{m-n}{2}}\widetilde{\chi} \otimes \nu ^{\frac{ra}{2}}\sigma_{n-ra,m-ra}\nu ^{\frac{-ra}{2}}\right),
\\& \rho \otimes \pi' \Big)\neq 0 .
\end{eqnarray*} 

Soit $b\geq 0$ maintenant maximal tel qu'il existe une repr\'esentation irr\'eductible $\rho'$ de $G'_{m-rb}$ avec $\pi'$ quotient de $$\nu ^{\frac{m-n}{2}}\widetilde{\chi} \times \dots \times \nu ^{\frac{m-n}{2}}\widetilde{\chi} \times \rho' $$ o\`u on a fait le produit de $b$ fois la repr\'esentation $ \nu ^{\frac{m-n}{2}}\widetilde{\chi} $.

Ceci \'equivaut, par le corollaire \ref{cambio} au fait que $\pi'$ soit sous-module de $$\rho' \times \nu ^{\frac{m-n}{2}}\widetilde{\chi} \times \dots \times \nu ^{\frac{m-n}{2}}\widetilde{\chi}, $$ o\`u on a fait le produit de $b$ fois la repr\'esentation $ \nu ^{\frac{m-n}{2}}\widetilde{\chi} $. 

Par \eqref{frob}, on a un homomorphisme non trivial de
$r_{m-rb,rb}^{G'_m}(\pi')$ dans $ \rho' \otimes \nu ^{\frac{m-n}{2}}\widetilde{\chi}\times \dots \times \nu ^{\frac{m-n}{2}}\widetilde{\chi} $ d'o\`u, par conjugaison, un homomorphisme non trivial de
$\overline{r}_{rb,m-rb}^{G'_{m}}(\pi')$ dans $\nu ^{\frac{m-n}{2}}\widetilde{\chi}\times \dots \times \nu ^{\frac{m-n}{2}}\widetilde{\chi} \otimes \rho'.$

D'un autre c\^ot\'e, le foncteur de Jacquet \'etant exact, on a un morphisme surjectif dans 
\begin{eqnarray*}
&\hspace{-.4cm}\Hom\Big(\overline{r}_{rb,m-rb}^{G'_{m}}\circ \ind_{P'_{ra,m-ra}}^{G'_m}\left(\nu ^{\frac{m-n}{2}}\widetilde{\chi}\! \times \!\dots \!\times \!\nu ^{\frac{m-n}{2}}\widetilde{\chi} \otimes \nu ^{\frac{ra}{2}}\sigma_{n-ra,m-ra}\nu ^{\prime \frac{-ra}{2}}\right),
\\& \rho \otimes \overline{r}_{rb,m-rb}^{G'_{m}}(\pi') \Big),
\end{eqnarray*} 
qui, compos\'e avec l'homomorphisme non trivial pr\'ec\'edent de
$\overline{r}_{rb,m-rb}^{G'_{m}}(\pi')$ dans $\nu ^{\frac{m-n}{2}}\widetilde{\chi}\times \dots \times \nu ^{\frac{m-n}{2}}\widetilde{\chi} \otimes \rho'$, montre que
\begin{eqnarray*}
&\hspace{-.4cm}\Hom\Big(\overline{r}_{rb,m-rb}^{G'_{m}} \circ \ind_{P'_{ra,m-ra}}^{G'_m}\left(\nu ^{\frac{m-n}{2}}\widetilde{\chi} \! \times \!\dots \!\times \! \nu ^{\frac{m-n}{2}}\widetilde{\chi} \otimes \nu ^{\frac{ra}{2}}\sigma_{n-ra,m-ra}\nu ^{\prime \frac{-ra}{2}}\right),
\\& \rho \otimes \nu ^{\frac{m-n}{2}}\widetilde{\chi}\times \dots \times \nu ^{\frac{m-n}{2}}\widetilde{\chi} \otimes \rho' \Big) \neq 0.
\end{eqnarray*} 

Par le lemme g\'eom\'etrique (\textit{cf.} \cite[\textsection 1.6]{Z1}) et la maximalit\'e de $b$, on d\'eduit que:
\begin{eqnarray*}\hspace{-.2cm}\Hom \hspace{-0.1cm} \Big( & \hspace{-0.3cm}\ind_{P'_{ra,rb-ra}}^{G'_{rb}}& \hspace{-0.5cm}\left(\nu ^{\frac{m-n}{2}}\widetilde{\chi} \! \times \!\dots \!\times \! \nu ^{\frac{m-n}{2}}\widetilde{\chi} \otimes \nu ^{\frac{ra}{2}}\overline{r}_{rb-ra,m-rb}^{G'_{m-ra}}\left(\sigma_{n-ra,m-ra}\right)\nu ^{\prime \frac{-ra}{2}}\right),
\\ && \rho \otimes \nu ^{\frac{m-n}{2}}\widetilde{\chi}\times \dots \times \nu ^{\frac{m-n}{2}}\widetilde{\chi} \otimes \rho' \Big)\neq 0 .
\end{eqnarray*}

D'apr\`es la proposition \ref{lema2}, il existe alors $i \in \left\{ 0, \dots , rb-ra \right\}$ tel que
$$\Hom\left( \overline{\tau'_i} ,\nu ^{\frac{m-n+ra}{2}}\widetilde{\chi}\times \dots \times \nu ^{\frac{m-n+ra}{2}}\widetilde{\chi} \otimes\nu ^{\frac{ra}{2}} \rho' \right)\neq 0 .$$
Le lemme suivant se montre comme le lemme \ref{lima}

\begin{lemma}
Les seuls $\overline{\tau'}_i$ qui peuvent avoir des quotients de la forme ci-dessus sont $\overline{\tau'}_{rb-ra}$ et $\overline{\tau'}_{rb-ra-1}$. Or, si $\nu ^{\frac{m-n}{2}}\widetilde{\chi}\neq \nu ^{\frac{m-2n-1}{2}}$, (\textit{i.e.} si $\chi \neq \nu ^{\frac{n+1}{2}}$) seul $\overline{\tau'}_{rb-ra}$ peut en avoir.
\end{lemma}

Ainsi on a, \`a nouveau, deux cas:

\underline{\bf Cas A.1} $\Hom\left( \overline{\tau'}_{rb-ra} ,\nu ^{\frac{m-n+ra}{2}}\widetilde{\chi}\times \dots \times \nu ^{\frac{m-n+ra}{2}}\widetilde{\chi} \otimes\nu ^{\frac{ra}{2}} \rho' \right)\neq 0 $ (ce qui arrive, en particulier, d'apr\`es le lemme pr\'ec\'edent si $\chi \neq \nu ^{\frac{2m-n+1}{2}}$ et $\chi \neq \nu ^{\frac{n+1}{2}}$).

\underline{\bf Cas A.2} $\Hom\left( \overline{\tau'}_{rb-ra-1} ,\nu ^{\frac{m-n+ra}{2}}\widetilde{\chi}\times \dots \times \nu ^{\frac{m-n+ra}{2}}\widetilde{\chi} \otimes\nu ^{\frac{ra}{2}} \rho' \right)\neq 0 $ et, dans ce cas, il faut que $\chi = \nu ^{\frac{n+1}{2}}$. 

Regardons d'abord le cas A.1.
\begin{lemma} \label{444}
Dans le cas A.1, on a $b=a$.
\end{lemma}
\begin{proof}

Si 
\begin{eqnarray*}\Hom \Big( \!\!& \ind_{P'_{ra,rb-ra}}^{G'_{rb}}&\left(\nu ^{\frac{m-n}{2}}\widetilde{\chi} \times \dots \times \nu ^{\frac{m-n}{2}}\widetilde{\chi} \otimes \nu ^{\frac{ra}{2}}\overline{\tau'}_{rb-ra}\nu ^{\prime \frac{-ra}{2}}\right),
\\ && \rho \otimes \nu ^{\frac{m-n}{2}}\widetilde{\chi}\times \dots \times \nu ^{\frac{m-n}{2}}\widetilde{\chi}\otimes \rho' \Big)\neq 0, 
\end{eqnarray*}
alors, par d\'efinition de $\overline{\tau'}_{rb-ra}$, on a aussi
\begin{eqnarray*}\Hom \Big( \!\!& \ind_{P_{n-rb,rb-ra}}^{G_{n-ra}}&\left( \nu ^{\frac{rb-ra}{2}}\sigma_{n-rb,m-rb}\nu ^{\prime \frac{-rb+ra}{2}}\otimes \nu ^{\frac{-ra}{2}}\chi \times \dots \times \nu ^{\frac{-ra}{2}}\chi \right),
\\ &&\nu ^{\frac{-ra}{2}} \rho \otimes\nu ^{\frac{ra}{2}} \rho' \Big)\neq 0, 
\end{eqnarray*}
o\`u on a fait le produit de $b-a$ fois la repr\'esentation $ \nu ^{\frac{-ra}{2}} \chi $, et donc, par maximalit\'e de $a$, il faut que $b=a$.
\end{proof}
Avant de passer aux autres cas, r\'esumons les r\'esultats obtenus en une proposition.
\begin{proposition}\label{calculo1} 
Soient $\pi \in \Irr G_n$, $\pi' \in \Irr G'_m$ telles que $\pi \otimes \pi'$ soit un quotient de $\sigma_{n,m}$. Soit aussi $$\chi \neq \begin{cases} \nu ^{\frac{n+1}{2}}\\ \nu ^{\frac{2m-n+1}{2}}\end{cases}$$ une repr\'esentation irr\'eductible cuspidale de $G_r$. Alors $a=b$ o\`u $a$ et $b$ sont d\'efinis par les conditions suivantes:
\begin{enumerate}
\item Il existe $\rho \in \Irr(G_{n-ra})$ avec $$\pi \hookrightarrow \chi \times \chi \times \dots \times \chi \times \rho,$$ o\`u on a fait le produit de $a$ fois la repr\'esentation cuspidale $\chi$ et $a$ est maximal.
\item Il existe $\rho' \in \Irr(G'_{m-rb})$ avec $$\pi' \hookrightarrow \rho' \times \nu ^{\frac{m-n}{2}}\widetilde{\chi} \times \dots \times \nu ^{\frac{m-n}{2}}\widetilde{\chi} ,$$ o\`u on a fait le produit de $b$ fois la repr\'esentation cuspidale $\nu ^{\frac{m-n}{2}}\widetilde{\chi}$ et $b$ est maximal.
\end{enumerate}
De plus, on a
$$\Hom \left( \sigma_{n-ra,m-ra}, \nu ^{\frac{-ra}{2}} \rho \otimes\nu ^{\frac{ra}{2}} \rho' \right) \neq 0.$$
\end{proposition}

\section{Unicit\'e}\label{unic}
La proposition pr\'ec\'edente, avec le th\'eor\`eme \ref{primerhowe}, nous permet de montrer l'unicit\'e de la correspondance th\^eta.
\begin{theorem}\label{uni}
Supposons $n \leq m$. Soit $\pi$ une repr\'esentation irr\'eductible de $G_{n}$. Il existe une unique repr\'esentation irr\'eductible $\pi'$ de $G'_{m}$ telle que
$$\Hom_{G_n \times G'_m}\left( \omega_{n,m}, \pi \otimes \pi'\right) \neq 0.$$ De plus, $\dim \left( \Hom_{G_n \times G'_m}\left( \omega_{n,m}, \pi \otimes \pi'\right) \right)=1$
\end{theorem}
\begin{proof}
Par r\'ecurrence on peut supposer que le th\'eor\`eme est vrai pour toute paire $\left(G_i, G'_j\right)$, o\`u $ij<nm$. Montrons-le pour la paire $\left(G_n, G'_m\right)$.

Soit $\pi' \in \Irr G'_m$ telles que $\pi \otimes \pi'$ soit un quotient de $\sigma_{n,m}$. Montrons que $\pi'$ est uniquement d\'etermin\'ee par $\pi$.

Supposons d'abord qu'il existe $\chi\neq \begin{cases} \nu ^{\frac{n+1}{2}}\\ \nu ^{\frac{2m-n+1}{2}}\end{cases}$ une repr\'esentation cuspidale de $G_r$, et $\tau \in \Irr(G_{n-r})$ avec $\pi \hookrightarrow \chi \times \tau.$ 

Soit $a>0$ et $\rho \in \Irr(G_{n-ra})$ avec $$\pi \hookrightarrow \chi \times \chi \times \dots \times \chi \times \rho,$$ o\`u on a fait le produit de $a$ fois la repr\'esentation cuspidale $\chi$ et $a$ est maximal.

D'apr\`es la proposition pr\'ec\'edente, il existe $\rho' \in \Irr(G'_{m-ra})$ avec $$\pi' \hookrightarrow \rho' \times \nu ^{\frac{m-n}{2}}\widetilde{\chi} \times \dots \times \nu ^{\frac{m-n}{2}}\widetilde{\chi} ,$$ o\`u on a fait le produit de $a$ fois la repr\'esentation cuspidale $\nu ^{\frac{m-n}{2}}\widetilde{\chi}$. De plus, on a
$$\Hom \left( \sigma_{n-ra,m-ra}, \nu ^{\frac{-ra}{2}} \rho \otimes\nu ^{\frac{ra}{2}} \rho' \right) \neq 0.$$

Par hypoth\`ese de r\'ecurrence, $\rho'$ est uniquement d\'etermin\'ee par $\rho$ et, par \cite[Th\'eor\`eme 5.1]{Min1},  $\pi'$ est l'\textit{unique} sous-module irr\'eductible de $\rho' \times \nu ^{\frac{m-n}{2}}\widetilde{\chi} \times \dots \times \nu ^{\frac{m-n}{2}}\widetilde{\chi}.$

Sinon, la repr\'esentation $\pi$ est telle que, si $\Jac_{\chi}(\pi) \neq 0$ et $\chi$ cuspidale, alors $\chi \in \left\{\nu ^{\frac{n+1}{2}}, \nu ^{\frac{2m-n+1}{2}} \right\}$. Dans ce cas $\pi$ n'appara\^it pas dans le bord de $\sigma_{n,m}$ car elle n'est pas quotient d'une repr\'esentation de la forme $\sharp\!-\!\ind^{G_{n}}_{\overline{P}_{n-k,k}}\left( 1_{n-k} \otimes \tau\right)$ avec $\tau \in \Irr (G_k)$.
En effet, si
$$\Hom \left( \sharp\!-\!\ind^{G_{n}}_{\overline{P}_{n-k,k}}\left( 1_{n-k} \otimes \tau\right) , \pi \right) \neq 0,$$
on trouve, apr\`es normalisation, que
$$\Hom \left( \ind^{G_{n}}_{\overline{P}_{n-k,k}}\left( \nu^{\frac{k}{2}} \otimes \nu^{\frac{k-n}{2}} \tau\right) , \pi \right) \neq 0.$$
Par conjugaison, on d\'eduit
$$\Hom \left( \ind^{G_{n}}_{P_{k,n-k}}\left( \nu^{\frac{k-n}{2}} \tau  \otimes \nu^{\frac{k}{2}}  \right) , \pi \right) \neq 0$$
puis, par \eqref{frobcas},
$$\Hom \left(\nu^{\frac{k-n}{2}} \tau  \otimes \nu^{\frac{k}{2}}   , \overline{r}_{P_{k,n-k}}^{G_n} (\pi) \right) \neq 0,$$
et, \`a nouveau par conjugaison,
$$\Hom \left( \nu^{\frac{k}{2}}  \otimes \nu^{\frac{k-n}{2}} \tau   , r_{P_{n-k,k}}^{G_n} (\pi) \right) \neq 0,$$
et donc $\Jac_{\nu ^{\frac{-n+2k+1}{2}}}(\pi) \neq 0$. Ceci n'est possible que si $k=n$ ou $k=m$ ce qui est absurde.

Ainsi d'apr\`es \ref{primerhowe}, $\pi'$ est l'unique quotient de $\sharp\!-\!\ind^{G'_{m}}_{P'_{m-n,n}}\left( 1_{m-n} \otimes \widetilde{\pi}\right)$.

\end{proof}

\section{La correspondance explicite}\label{lademo}

Soient $\pi\in \Irr(G_n), \pi' \in \Irr(G'_m)$ telles que
$$\Hom_{G_n \times G'_m}\left( \omega_{n,m}, \pi \otimes \pi'\right) \neq 0.$$
La fin de l'article est consacr\'e au calcul des param\`etres de Langlands $\pi'$ en termes de ceux de $\pi$.

Soient $\tau_1,\dots ,\tau_N$ des repr\'esentations essentiellement de carr\'e int\'egrable et $\alpha_1, \dots, \alpha_N \in \mathbb{R}$ tels que, pour tout $1 \leq i \leq N$, $\nu^{\alpha_i}\tau_i $ soit une repr\'esentation de carr\'e int\'egrable. Soit $\sigma$ une permutation de $\left\{1, \dots, N \right\}$ telle que
$\alpha_{\sigma(i)} \geq \alpha_{\sigma(j)}$ si $i <j$. La repr\'esentation $\tau_{\sigma(1)} \times \tau_{\sigma(2)} \times \dots \times \tau_{\sigma(N)}$ a un unique quotient irr\'eductible, et on dira que c'est le quotient de Langlands de $\tau_1 \times \dots \times \tau_N$.

Supposons que $\pi$ est le quotient de Langlands de $\tau_1 \times \dots \times \tau_N$, o\`u $\tau_1,\dots ,\tau_N$ sont des repr\'esentations essentiellement de carr\'e int\'egrable. Notons alors $\theta^\ast _m(\pi)$ le quotient de Langlands de $$\nu ^{\frac{m-2n-1}{2}} \times \dots \times \nu ^{\frac{-m+1}{2}} \times\nu ^{\frac{m-n}{2}} \widetilde{\tau_1} \times \dots \times \nu ^{\frac{m-n}{2}} \widetilde{\tau_N}.$$

\begin{theorem}\label{expl}
Si $m \geq n$ et $$\Hom_{G_n \times G'_m}\left( \sigma_{n,m}, \pi \otimes \pi'\right) \neq 0,$$
alors $\pi'= \theta^\ast _m(\pi)$.
\end{theorem}
\begin{remark}
En particulier, si $m=n$, $\pi'= \widetilde{\pi}$, comme on l'avait d\'ej\`a montr\'e au d\'ebut de la section \ref{filtraciones1}. Dans la preuve, on supposera alors $m>n$.
\end{remark}

Des th\'eor\`emes \ref{uni} et \ref{expl}, il r\'esulte, avec la normalisation correspondante, un th\'eor\`eme similaire pour la repr\'esentation $\omega_{n,m}$ (\textit{cf.} \eqref{omega}):

\begin{corollary}\label{expl2}
Soit $\pi$ une repr\'esentation irr\'eductible de $G_{n}$.
\begin{enumerate}
\item Si $\Hom_{G_n}\left( \omega_{n,m}, \pi \right) \neq 0$, alors il existe une unique repr\'esentation irr\'eductible $\pi'$ de $G'_{m}$ telle que
$$\Hom_{G_n \times G'_m}\left( \omega_{n,m}, \pi \otimes \pi'\right) \neq 0.$$ De plus, $\dim \left( \Hom_{G_n \times G'_m}\left( \omega_{n,m}, \pi \otimes \pi'\right) \right)=1$.
\item Supposons $n \leq m$. Alors  $\Hom_{G_n}\left( \omega_{n,m}, \pi \right) \neq 0$ et, si $\pi$ est le quotient de Langlands de $\tau_1 \times \dots \times \tau_N$, o\`u $\tau_1,\dots ,\tau_N$ sont des repr\'esentations essentiellement de carr\'e int\'egrable alors  $\pi'$ est le quotient de Langlands de $$\nu ^{-\frac{m-n-1}{2}} \times \dots \times \nu ^{\frac{m-n-1}{2}} \times \widetilde{\tau_1} \times \dots \times \widetilde{\tau_N}.$$
\end{enumerate}
\end{corollary}
Pour la preuve du th\'eor\`eme \ref{expl}, on va utiliser plusieurs fois le lemme suivant \cite[Th\'eor\`eme A.3]{Min1}
\begin{lemma}\label{rrec}
Soient $\chi$ une repr\'esentation cuspidale de $G_r$ telle que $\chi \neq \begin{cases} \nu ^{\frac{n+1}{2}}\\ \nu ^{\frac{2m-n+1}{2},}\end{cases}$ $\rho\in \Irr(G_{n-r})$, et $\pi$ l'unique sous-repr\'esentation irr\'eductible de $\chi \times \rho$. Notons $\pi'$ l'unique sous-repr\'e\-sen\-tation irr\'eductible de $\nu^{\frac{-r}{2}}\theta_{m-r}^\ast(\nu^{\frac{-r}{2}}\rho) \times \nu^{\frac{m-n}{2}}\widetilde{\chi}$. Alors
$$\pi'= \theta_{m}^\ast(\pi).$$
\end{lemma}
\begin{corollary}\label{rrrec}
Soient $\chi$ une repr\'esentation cuspidale de $G_r$, $\chi \neq \begin{cases} \nu ^{\frac{n+1}{2}}\\ \nu ^{\frac{2m-n+1}{2}.}\end{cases}$ Soient $a\in \mathbb{N}^\ast$, $\rho\in \Irr(G_{n-ra})$, et $\pi$ l'unique sous-repr\'e\-sen\-tation irr\'eductible de $\chi \times \dots \times \chi \times \rho$ o\`u on a fait le produit de $a$ fois la repr\'esentation $\chi$. Notons $\pi'$ l'unique sous-repr\'e\-sen\-tation irr\'eductible de $\nu^{\frac{-ra}{2}}\theta_{m-ra}^\ast(\nu^{\frac{-ra}{2}}\rho) \times \nu ^{\frac{m-n}{2}}\widetilde{\chi} \times \dots \times \nu ^{\frac{m-n}{2}}\widetilde{\chi} ,$ o\`u on a fait le produit de $a$ fois la repr\'esentation $\nu ^{\frac{m-n}{2}}\widetilde{\chi}$. Alors
$$\pi'= \theta_{m}^\ast(\pi).$$
\end{corollary}
\begin{proof}
Par r\'ecurrence sur $a$. Si $a=1$, c'est le lemme \ref{rrec}.
Supposons $a>1$. Notons $\pi_1$ l'unique sous-repr\'esentation irr\'eductible de $\chi \times \dots \times \chi \times \rho$ o\`u on a fait le produit de $a-1$ fois la repr\'esentation $\chi$. Notons $\pi'_1$ l'unique sous-repr\'esentation irr\'eductible de $\nu^{\frac{-ra}{2}}\theta_{m-ra}^\ast(\nu^{\frac{-ra}{2}}\rho) \times \nu ^{\frac{m-n}{2}}\widetilde{\chi} \times \dots \times \nu ^{\frac{m-n}{2}}\widetilde{\chi} ,$ o\`u on a fait le produit de $a-1$ fois la repr\'esentation $\nu ^{\frac{m-n}{2}}\widetilde{\chi}$. Alors, par hypoth\`ese de r\'ecurrence
$$\pi'_1= \nu^{\frac{-r}{2}}\theta_{m-r}^\ast(\nu^{\frac{-r}{2}}\rho)$$
et, de plus $\pi$ est l'unique sous-repr\'esentation irr\'eductible de $\chi \times \pi_1$ et $\pi'$ est l'unique sous-repr\'esentation irr\'eductible de $\pi'_1 \times \nu^{\frac{m-n}{2}}\widetilde{\chi}$. Le r\'esultat d\'ecoule, \`a nouveau, du lemme \ref{rrec}.
\end{proof}

\section{Le cas simple}\label{simple}

Supposons qu'il existe $\chi\neq \begin{cases} \nu ^{\frac{n+1}{2}}\\ \nu ^{\frac{2m-n+1}{2}}\end{cases}$ une repr\'esentation cuspidale de $G_r$ et un entier strictement positif $a$, tels que l'une des conditions suivantes \'equivalentes (par la proposition \ref{calculo1}) soit satisfaite:

\begin{enumerate}
\item Il existe $\rho \in \Irr(G_{n-ra})$ avec 
\begin{equation}\label{rec}
\pi \hookrightarrow \chi \times \chi \times \dots \times \chi \times \rho,
\end{equation}
o\`u on a fait le produit de $a$ fois la repr\'esentation $\chi$, $a$ maximal.
\item ou bien, il existe $\rho' \in \Irr(G'_{m-ra})$ avec 
\begin{equation}\label{rec2}
\pi' \hookrightarrow \rho' \times \nu ^{\frac{m-n}{2}}\widetilde{\chi} \times \dots \times \nu ^{\frac{m-n}{2}}\widetilde{\chi} ,
\end{equation}
o\`u on a fait le produit de $a$ fois la repr\'esentation $\nu ^{\frac{m-n}{2}}\widetilde{\chi}$, $a$ maximal.
\end{enumerate}
Puisque, par la proposition \ref{calculo1},
$$ \Hom \left( \sigma_{n-ra,m-ra}, \nu ^{\frac{-ra}{2}} \rho \otimes\nu ^{\frac{ra}{2}} \rho' \right) \neq0, $$
et $a>0$, on peut supposer, par hypoth\`ese de r\'ecurrence, que 
\begin{equation}\label{rec3}
\rho' =\nu ^{\frac{-ra}{2}} \theta^\ast _{m-ra}\left( \nu ^{\frac{-ra}{2}}\rho\right)\end{equation}

On d\'eduit de \eqref{rec}, \eqref{rec2} et \eqref{rec3}, gr\^ace au corollaire \ref{rrrec}, que $\pi'= \theta^\ast _m(\pi)$.
\section{Sur les param\`etres de Zelevinsky}
Soit $r \in \mathbb{R}, n \in \mathbb{N}^{\ast}$. On dit que la suite $\Delta=\left\{r, r+1,\dots, r+n-1 \right\}$ de nombres r\'eels est un segment et l'ensemble de tous les segments sera not\'e $S$. Il est muni d'une action de $\mathbb{R}$ d\'efinie par $$k \left\{r, r+1,\dots, r+n-1 \right\}=\left\{k+r, k+r+1,\dots, k+r+n-1 \right\}.$$

On notera aussi 
\begin{eqnarray*}
b\left(\left\{r, r+1,\dots, r+n-1 \right\} \right) &=&r , \\ e\left( \left\{r, r+1,\dots, r+n-1 \right\} \right) &=&r+n-1.
\end{eqnarray*}
les extr\'emit\'es du segment. On note $l\left(\left\{r, r+1,\dots, r+n-1 \right\} \right)=n$ sa longueur.

On d\'efinit un pr\'eordre sur $S$ par $\Delta \leq \Delta'$ si $b\left( \Delta \right) \leq b\left( \Delta' \right)$. On note $M(S)$ l'ensemble de multisegments, \textit{i.e,} des fonctions $m:S \rightarrow \mathbb{N}$ \`a support fini (on pensera \`a un ensemble de segments compt\'es avec multiplicit\'es). Un multisegment $\Delta_1, \Delta_2 ,\dots ,\Delta_N$ est dit rang\'e si $\Delta_N \leq \dots \leq \Delta_2 \leq \Delta_1$.

A chaque segment $\Delta =\left\{r, r+1,\dots, r+n-1 \right\}$ on associe une repr\'esentation irr\'eductible de ${\rm GL}_n(D)$, not\'ee $\left< \Delta \right>^t$, d\'efinie comme l'unique quotient irr\'eductible de $\nu^r \times \nu^{r+1} \times \dots \times \nu^{r+n-1}$. La repr\'esentation $\left< \Delta \right>^t$ est essentiellement de carr\'e int\'egrable. La contragr\'ediente $\widetilde{\left< \Delta \right>^t}$ est la repr\'esentation $\left< \widetilde{ \Delta} \right>^t$ o\`u $\widetilde{ \Delta}= \left\{ -r-n+1, -r-n +2, \dots -r \right\}$.

A chaque multisegment $\Delta_1, \Delta_2 ,\dots ,\Delta_N$ on associe une repr\'esentation irr\'eductible de ${\rm GL}_{\sum l\left(\Delta_i\right)}(D)$, not\'ee $\left< \Delta_1, \Delta_2 ,\dots ,\Delta_N \right>^t$, d\'efinie comme l'unique quotient irr\'eductible de $\left< \Delta_{\sigma(1)} \right>^t \times \left< \Delta_{\sigma(2)} \right>^t \times \dots \times \left< \Delta_{\sigma(N)} \right>^t$, o\`u $\sigma$ est une permutation de l'ensemble $\left\{1, \dots , N \right\}$ telle que le multisegment $ \Delta_{\sigma(1)} ,  \Delta_{\sigma(2)}, \dots , \Delta_{\sigma(N)} $ soit rang\'e. La contragr\'ediente $\widetilde{\left< \Delta_1, \Delta_2 ,\dots ,\Delta_N \right>^t}$ est la repr\'esentation $\left< \widetilde{\Delta_1}, \widetilde{\Delta_2} ,\dots ,\widetilde{\Delta_N} \right>^t$.

La proposition suivante, dans le cas $D=F$ est montr\'e dans \cite[6.9.]{Z1}. Dans le cas o\`u $D \neq F$ la preuve est analogue et se trouve dans \cite[2.3.7]{Minthe}. 
\begin{proposition}\label{segm}
Soit $\Delta_1, \Delta_2 ,\dots ,\Delta_N$ un multiensegment rang\'e, avec $\Delta_1 =\left\{b, b+1,\dots, e \right\}$ et $\Delta_N =\left\{b', b'+1,\dots, e' \right\}$, alors 
\begin{enumerate}
\item Si
$\Jac_{\nu^l}\left(\left< \Delta_1, \Delta_2 ,\dots ,\Delta_N \right>^t\right) \neq 0$ on a que $l \geq e'$.
\item Si
$\overline{\Jac}_{\nu^l}\left(\left< \Delta_1, \Delta_2 ,\dots ,\Delta_N \right>^t\right) \neq 0$ on a que $l \leq b$.
\end{enumerate}
\end{proposition}
\section{Fin de la preuve}\label{fin}
On s'est ramen\'e aux cas o\`u $\pi$ et $\pi'$ sont des repr\'esentations tr\`es particuli\`eres, des repr\'esentations v\'erifiant les propri\'et\'es suivantes:
\begin{enumerate}
\item[J.1] Si $\Jac_{\chi}(\pi) \neq 0$ et $\chi$ cuspidale, alors $\chi \in \left\{\nu ^{\frac{n+1}{2}}, \nu ^{\frac{2m-n+1}{2}} \right\}$,
\item[J.2] si $\overline{\Jac}_{\chi}(\pi') \neq 0$ et $\chi$ cuspidale, alors $\chi \in \left\{\nu ^{\frac{-m-1}{2}}, \nu ^{\frac{m-2n-1}{2}} \right\}$.
\end{enumerate}
On va utiliser les propri\'et\'es de la section pr\'ec\'edente pour terminer la preuve du th\'eor\`eme \ref{expl}. On rappelle qu'on suppose $m>n$.

Soit $\pi \in \Irr(G_n)$ telle que $\Jac_{\nu ^{\frac{n+1}{2}}}(\pi) \neq 0$. Soit $a$ maximal tel que $$\pi \hookrightarrow \chi \times \chi \times \dots \times \chi \times \rho,$$ o\`u on a fait le produit de $a$ fois le caract\`ere $\chi=\nu ^{\frac{n+1}{2}}$. On se trouve, avec les notations de la section \ref{application}, dans le cas A. On rappelle que l'on a deux possibilit\'es, A1 et A2:

\underline{\bf Cas A.1}  Ou bien, il existe $\rho' \in \Irr(G'_{m-a})$ avec $$\pi' \hookrightarrow \rho' \times \nu ^{\frac{m-n}{2}}\chi^{-1} \times \dots \times \nu ^{\frac{m-n}{2}}\chi^{-1} ,$$ o\`u on a fait le produit de $a$ fois le caract\`ere $\nu ^{\frac{m-n}{2}}\chi^{-1}$, $a$ maximal et $$\Hom \left( \sigma_{n-a,m-a}, \nu ^{\frac{-a}{2}} \rho \otimes\nu ^{\frac{a}{2}} \rho' \right) \neq 0$$
(proposition \ref{calculo1});

Par hypoth\`ese de r\'ecurrence, on a que:
\begin{equation}\label{premiercas}
\pi' \hookrightarrow \rho' \times \nu ^{\frac{m-2n-1}{2}} \times \dots \times \nu ^{\frac{m-2n-1}{2}}
\end{equation}
o\`u 
\begin{eqnarray*}
\rho'&=&\nu^{\frac{-a}{2}}\theta_{m-a}^\ast(\nu^{\frac{-a}{2}}\rho)\\&=&\left< \nu ^{\frac{m-2n-1}{2}}, \dots, \nu ^{\frac{-m+1}{2}} ,\nu^{\frac{m-n}{2}} \widetilde{\Delta_1}, \dots, \nu^{\frac{m-n}{2}} \widetilde{\Delta_N}\right>^t \end{eqnarray*}
si $\rho=\left< \Delta_1, \dots, \Delta_N\right>^t$, avec les notations de la section pr\'ec\'edente.

\begin{lemma}
Il n'existe pas des repr\'esentations irr\'eductibles $\pi$ et $\pi'$ satisfaisant aux conditions J.1, J.2 et \eqref{premiercas}.
\end{lemma}
\begin{proof}
D'apr\`es J.1, on sait que, si $\Jac_{\chi}(\pi) \neq 0$ et $\chi$ cuspidale, alors $\chi \in \left\{\nu ^{\frac{n+1}{2}}, \nu ^{\frac{2m-n+1}{2}} \right\}$. Par \ref{segm}.(1), tous les segments de $\rho$ finissent alors par $e_i \geq \frac{n+1}{2}$. Ainsi, par d\'efinition de $\rho'$, tous les segments de $\rho'$ commencent alors par $b'_i \leq \frac{m-2n-1}{2}$. De plus, puisque $m \neq n$, $\left\{\nu ^{\frac{m-2n-1}{2}} \right\}$ est un segment de $\rho'$ et donc, par \cite[Th\'eor\`eme 6.6.(3)]{Min1}, il existe $\tau' \in \Irr(G'_{m-a-1})$  tel que $\rho' \hookrightarrow \tau' \times \nu ^{\frac{m-2n-1}{2}}$ ce qui contredit la maximalit\'e de $a$.
\end{proof}

\underline{\bf Cas A.2} Alors on a bien (avec les notations de \ref{calculo1}) $$\Hom\left( \overline{\tau'}_{b-a-1} ,\nu ^{\frac{m-n+a}{2}}\widetilde{\chi}\times \dots \times \nu ^{\frac{m-n+a}{2}}\widetilde{\chi} \otimes\nu ^{\frac{a}{2}} \rho' \right)\neq 0. $$ Dans ce cas on a une proposition similaire \`a la proposition \ref{calculo1}
\begin{proposition}
Soient $\pi \in \Irr G_n$, $\pi' \in \Irr G'_m$ telles que $\pi \otimes \pi'$ soit un quotient de $\sigma_{n,m}$ satisfaisant aux conditions J.1 et J.2. Soit aussi $\chi= \nu ^{\frac{n+1}{2}}$ un caract\`ere de $G_1$. Alors $a=b-1$ o\`u $a$ et $b$ sont d\'efinis par les conditions suivantes:
\begin{enumerate}
\item Il existe $\rho \in \Irr(G_{n-a})$ avec $$\pi \hookrightarrow \chi \times \chi \times \dots \times \chi \times \rho,$$ o\`u on a fait le produit de $a$ fois le caract\`ere $\chi$ et $a$ est maximal.
\item Il existe $\rho' \in \Irr(G'_{m-b})$ avec $$\pi' \hookrightarrow \rho' \times \nu ^{\frac{m-n}{2}}\widetilde{\chi} \times \dots \times \nu ^{\frac{m-n}{2}}\widetilde{\chi} ,$$ o\`u on a fait le produit de $b$ fois le caract\`ere $\nu ^{\frac{m-n}{2}}\widetilde{\chi}$ et $b$ est maximal.
\end{enumerate}
De plus, on a
$$\Hom \left( \sigma_{n-a,m-a-1}, \nu ^{\frac{-a}{2}-1} \rho \otimes\nu ^{\frac{a+1}{2}} \rho' \right) \neq 0.$$
\end{proposition}
\begin{proof}
On a que
\begin{eqnarray*}\Hom \Big( \!\!& \ind_{P'_{a,b-a}}^{G_{b}}&\left(\nu ^{\frac{m-n}{2}}\widetilde{\chi} \times \dots \times \nu ^{\frac{m-n}{2}}\widetilde{\chi} \otimes \nu ^{\frac{a}{2}}\overline{\tau'}_{b-a-1}\nu ^{\frac{-a}{2}}\right),
\\ && \rho \otimes \nu ^{\frac{m-n}{2}}\widetilde{\chi}\times \dots \times \nu ^{\frac{m-n}{2}}\widetilde{\chi}\otimes \rho' \Big)\neq 0 ,
\end{eqnarray*}
d'o\`u, par d\'efinition de $\overline{\tau'}_{b-a-1}$,
\begin{eqnarray*}\Hom \Big( \!\!& \ind_{P_{n-b-1,b-a-1}}^{G_{n-a}}&\left( \nu ^{\frac{b+1}{2}}\sigma_{n-b+1,m-b}\nu ^{\prime \frac{-b}{2}}\otimes \nu ^{\frac{-a}{2}}\chi \times \dots \times \nu ^{\frac{-a}{2}}\chi \right),
\\ && \rho \otimes \rho' \Big)\neq 0, 
\end{eqnarray*}
o\`u on a fait le produit de $b-a-1$ fois le caract\`ere $ \nu ^{\frac{-a}{2}} \chi $, et donc, par maximalit\'e de $a$, on a $b=a+1$.

\end{proof}

Ainsi, il existe $\rho' \in \Irr(G'_{m-a-1})$ avec $$\pi' \hookrightarrow \rho' \times \nu ^{\frac{m-n}{2}}\chi^{-1} \times \dots \times \nu ^{\frac{m-n}{2}}\chi^{-1} ,$$ o\`u on a fait le produit de $a+1$ fois le caract\`ere $\nu ^{\frac{m-n}{2}}\chi^{-1}$, $a$ maximal et $$\Hom \left( \sigma_{n-a,m-a-1}, \nu ^{\frac{-a}{2}-1} \rho \otimes\nu ^{\frac{a+1}{2}} \rho' \right) \neq 0.$$
Alors, par hypoth\`ese de r\'ecurrence on a 
\begin{equation*}
\dim \left( \Hom \left( \sigma_{n-a,m-a-1}, \nu ^{\frac{-a}{2}-1} \rho \otimes\nu ^{\frac{a+1}{2}} \rho' \right) \right)=1,\end{equation*}
et
\begin{eqnarray*}
 \rho'&=&\nu^{\frac{-a-1}{2}}\theta_{m-a}^\ast(\nu^{\frac{-a}{2}-1}\rho)\\&=& \left< \nu ^{\frac{m-2n-3}{2}}, \dots, \nu ^{\frac{-m+1}{2}} ,\nu^{\frac{m-n}{2}} \widetilde{\Delta_1}, \dots, \nu^{\frac{m-n}{2}} \widetilde{\Delta_N}\right>^t.
\end{eqnarray*}
si $\rho=\left< \Delta_1, \dots, \Delta_N\right>^t$.
\begin{lemma}
Supposons que $\pi$ et $\pi'$ sont deux repr\'esentation irr\'eductibles qui satisfont aux conditions J.1, J.2 et telles que $\pi$ est un sous-module de $\nu ^{\frac{n+1}{2}} \times \dots \times \nu ^{\frac{n+1}{2}} \times \rho$, o\`u on a fait le produit de $a$ fois le caract\`ere $\nu ^{\frac{n+1}{2}}$, $\pi'$ est un sous-module de $ \rho' \times \nu ^{\frac{m-2n-1}{2}} \times \dots \times \nu ^{\frac{m-2n-1}{2}}$ o\`u on a fait le produit de $a+1$ fois le caract\`ere $\nu ^{\frac{m-2n-1}{2}}\chi^{-1}$ et 
$$\rho'=\left< \nu ^{\frac{m-2n-3}{2}}, \dots, \nu ^{\frac{-m+1}{2}} ,\nu^{\frac{m-n}{2}} \widetilde{\Delta_1}, \dots, \nu^{\frac{m-n}{2}} \widetilde{\Delta_N}\right>^t$$
si $\rho=\left< \Delta_1, \dots, \Delta_N\right>^t$.

Alors $\pi'=\theta_{m}^\ast(\pi)$.

\end{lemma}
\begin{proof}
Notons $\pi'_1$ l'unique sous-repr\'esentation irr\'eductible de 
\begin{eqnarray*}
\rho' \times \nu ^{\frac{m-2n-1}{2}} \times \dots \times \nu ^{\frac{m-2n-1}{2}},\end{eqnarray*}
o\`u on a fait le produit de $a$ fois le caract\`ere $\nu ^{\frac{m-2n-1}{2}}$, alors, d'apr\`es le corollaire \ref{rrrec}, on a que 
\begin{eqnarray*}
\pi'_1&=&\nu^{\frac{-a}{2}}\theta_{m-a-1}^\ast(\nu^{\frac{-a}{2}}\pi)\\&=&\left< \nu ^{\frac{m-2n-3}{2}}, \dots, \nu ^{\frac{-m+1}{2}} ,\nu^{\frac{m-n}{2}} \widetilde{\Delta'_1}, \dots, \nu^{\frac{m-n}{2}} \widetilde{\Delta'_N}\right>^t\end{eqnarray*} si $\pi=\left< \Delta'_1, \dots, \Delta'_N\right>^t$.

De plus, $\pi'$ est l'unique sous-repr\'esentation irr\'eductible de $\pi'_1 \times \nu ^{\frac{m-2n-1}{2}}$. Voyons finalement que $\pi'=\theta_{m}^\ast(\pi)$.

Puisque d'apr\`es J.1, on sait que, si $\Jac_{\chi}(\pi) \neq 0$ et $\chi$ cuspidale, alors $\chi \in \left\{\nu ^{\frac{n+1}{2}}, \nu ^{\frac{2m-n+1}{2}} \right\}$. Par la proposition \ref{segm}(1), tous les segments de la forme $\Delta'$ finissent alors par $e_i \geq \frac{n+1}{2}$. Ainsi, tous les segments de la forme $\nu^{\frac{m-n}{2}} \widetilde{\Delta'_i}$ commencent par $b'_i \leq \frac{m-2n-1}{2}$ et, a fortiori, tous les segments de $\pi'_1$ commencent par $b'_i \leq \frac{m-2n-1}{2}$.

Ainsi, par \cite[Th\'eor\`eme 6.6(.2)]{Min1}, on a  que
$$\pi'=\left< \nu ^{\frac{m-2n-1}{2}},  \nu ^{\frac{m-2n-3}{2}}, \dots, \nu ^{\frac{-m+1}{2}} ,\nu^{\frac{m-n}{2}} \widetilde{\Delta_1}, \dots, \nu^{\frac{m-n}{2}} \widetilde{\Delta_N}\right>^t=\theta_{m}^\ast(\pi) .
$$
\end{proof}

Il ne nous reste maintenant \`a traiter que les cas o\`u $\pi$ et $\pi'$ sont des repr\'esentations v\'erifiant les propri\'et\'es suivantes:
\begin{enumerate}
\item[H.1] Si $\Jac_{\chi}(\pi) \neq 0$ et $\chi$ cuspidale, alors $\chi = \nu ^{\frac{2m-n+1}{2}}$,
\item[H.2] si $\overline{\Jac}_{\chi}(\pi') \neq 0$ et $\chi$ cuspidale, alors $\chi =\nu ^{\frac{-m-1}{2}}$.
\end{enumerate}

\begin{proposition}
Il n'existe pas de repr\'esentations irr\'eductibles $\pi\in \Irr(G_n)$ et $\pi'\in\Irr(G'_m)$, $m>n$, satisfaisant aux conditions H.1 et H.2 et telles que $\Hom(\sigma_n, \pi \otimes \pi')\neq 0$
\end{proposition}


Soit $\pi \in \Irr(G_n)$ v\'erifiant H.1 et soit $a$ maximal tel que $$\pi \hookrightarrow \chi \times \chi \times \dots \times \chi \times \rho,$$ o\`u on a fait le produit de $a$ fois le caract\`ere $\chi=\nu ^{\frac{2m-n+1}{2}}$. On a alors, avec les notations de la section \ref{application}, que
\begin{enumerate}
\item[{\bf Cas A}] Ou bien, $\Hom\left( \tau_{a}, \chi \times \chi \times \dots \times \chi \otimes \rho\otimes \pi' \right)\neq 0,$
\item[{\bf Cas B}] ou bien, $\Hom\left( \tau_{a-1}, \chi \times \chi \times \dots \times \chi \otimes \rho\otimes \pi' \right)\neq 0.$
\end{enumerate}

C'est \`a dire:

\underline{\bf Cas A}  Puisque $m\neq n$ on est bien dans le cas A.1 et donc, d'apr\`es \ref{calculo1}, il existe $\rho' \in \Irr(G'_{m-a})$ avec $$\pi' \hookrightarrow \rho' \times \nu ^{\frac{m-n}{2}}\chi^{-1} \times \dots \times \nu ^{\frac{m-n}{2}}\chi^{-1} ,$$ o\`u on a fait le produit de $a$ fois le caract\`ere $\nu ^{\frac{m-n}{2}}\chi^{-1}$, $a$ maximal et $$\Hom \left( \sigma_{n-a,m-a}, \nu ^{\frac{-a}{2}} \rho \otimes\nu ^{\frac{a}{2}} \rho' \right) \neq 0.$$

Ceci n'est pas possible. En effet, par r\'ecurrence, on a que:
$$\pi' \hookrightarrow \rho' \times \nu ^{\frac{-m-1}{2}} \times \dots \times \nu ^{\frac{-m-1}{2}},$$ et $$\rho'=\left< \nu ^{\frac{m-2n-1}{2}}, \dots, \nu ^{\frac{-m+1}{2}} ,\nu^{\frac{m-n}{2}} \widetilde{\Delta_1}, \dots, \nu^{\frac{m-n}{2}} \widetilde{\Delta_N}\right>^t
$$ si $\rho=\left< \Delta_1, \dots, \Delta_N\right>^t$.
 
Il y a ainsi des segments de $\pi'$ commen\c{c}ant par $x_i > \frac{-m-1}{2}$ ce qui, par la proposition \ref{segm}(2), contredit H.2.

\underline{\bf Cas B}  Sinon montrons qu'il existe $\rho' \in \Irr(G'_{m-a+1})$ avec $$\pi' \hookrightarrow \rho' \times \nu ^{\frac{m-n}{2}}\chi^{-1} \times \dots \times \nu ^{\frac{m-n}{2}}\chi^{-1} ,$$ o\`u on a fait le produit de $a-1$ fois le caract\`ere $\nu ^{\frac{m-n}{2}}\chi^{-1}$, $a$ maximal et $$\Hom \left( \sigma_{n-a,m-a+1}, \nu ^{\frac{-a}{2}} \rho \otimes\nu ^{\frac{a+1}{2}} \rho' \right) \neq 0.$$

En effet, 
$\Hom\left( \tau_{a-1}, \chi \times \chi \times \dots \times \chi \otimes \rho\otimes \pi' \right)\neq 0 $ implique que
\begin{eqnarray*}
\Hom\Big( \ind_{P'_{a-1,m-a+1}}^{G'_m}&&\hspace{-0.8cm}\Big(\nu ^{\frac{m-n}{2}}\chi^{-1} \times \dots \times \nu ^{\frac{m-n}{2}}\chi^{-1} \otimes \\ &&\nu ^{\frac{a}{2}}\sigma_{n-a,m-a+1}\nu ^{\prime \frac{-a-1}{2}}\Big),
\rho \otimes \pi' \Big)\neq 0 .
\end{eqnarray*} 

Soit $b$ maintenant maximal tel qu'il existe une repr\'esentation irr\'eductible $\rho'$ de $G'_{m-b}$ avec $\pi'$ quotient de $$\nu ^{\frac{m-n}{2}}\chi^{-1} \times \dots \times \nu ^{\frac{m-n}{2}}\chi^{-1} \times \rho' $$ o\`u on a fait le produit de $b$ fois le caract\`ere $ \nu ^{\frac{m-n}{2}}\chi^{-1} $. D'apr\`es le corollaire \ref{cambio} on a une fl\`eche non nulle
$$\overline{r}_{b,m-b}^{G'_m}(\pi') \hookrightarrow \nu ^{\frac{m-n}{2}}\chi^{-1} \times \dots \times \nu ^{\frac{m-n}{2}}\chi^{-1} \otimes \rho'.$$
Puisque $\chi \neq \nu ^{\frac{n+1}{2}}$, on montre, comme dans \ref{calculo1}, que $b=a-1$ et donc
$$\Hom \left( \sigma_{n-a,m-a+1}, \nu ^{\frac{-a}{2}} \rho \otimes\nu ^{\frac{a+1}{2}} \rho' \right) \neq 0.$$

Ainsi, il existe $\rho' \in \Irr(G'_{m-a+1})$ avec $$\pi' \hookrightarrow \rho' \times \nu ^{\frac{m-n}{2}}\chi^{-1} \times \dots \times \nu ^{\frac{m-n}{2}}\chi^{-1} ,$$ o\`u on a fait le produit de $a-1$ fois le caract\`ere $\nu ^{\frac{m-n}{2}}\chi^{-1}$, $a$ maximal et $$\Hom \left( \sigma_{n-a,m-a+1}, \nu ^{\frac{-a}{2}} \rho \otimes\nu ^{\frac{a+1}{2}} \rho' \right) \neq 0.$$
Alors, par hypoth\`ese de r\'ecurrence on a $$\pi' \hookrightarrow \rho' \times \nu ^{\frac{-m-1}{2}} \times \dots \times \nu ^{\frac{-m-1}{2}},$$ et 
$$ \rho'=\left< \nu ^{\frac{m-2n-1}{2}}, \dots, \nu ^{\frac{-m-1}{2}} ,\nu^{\frac{m-n}{2}} \widetilde{\Delta_1}, \dots, \nu^{\frac{m-n}{2}} \widetilde{\Delta_N}\right>^t
$$ si $\rho=\left< \Delta_1, \dots, \Delta_N\right>^t$.
Si $n \neq m$, on trouve ainsi des segments de $\pi'$ commen\c{c}ant par $x_i > \frac{-m-1}{2}$ ce qui, \`a nouveau par \ref{segm}(2), contredit H.2.

\end{document}